\documentclass[11pt,twoside]{amsart}
\usepackage{amsxtra}
\usepackage{color}
\usepackage{amsopn}
\usepackage{amsmath,amsthm,amssymb,arydshln}
\usepackage{mathrsfs,mathtools}
\usepackage{stmaryrd}
\usepackage{hyperref}
\usepackage{enumerate}
\usepackage{xcolor}
\usepackage{pifont}
\usepackage{adjustbox}

\newtheorem{theorem}{Theorem}[section]
\newtheorem{corollary}[theorem]{Corollary}
\newtheorem{proposition}[theorem]{Proposition}
\newtheorem{lemma}[theorem]{Lemma}
\newtheorem*{theorem*}{Theorem}

\theoremstyle{definition}
\newtheorem{definition}[theorem]{Definition}

\newtheorem{remark}[theorem]{Remark}

\makeatletter
\@namedef{subjclassname@2020}{%
  \textup{2020} Mathematics Subject Classification}
\makeatother

\newcommand{\R}{{\mathbb R}}

\newcommand{\C}{{\mathbb C}}

\newcommand{\psip}{\psi^{\sst+}}
\newcommand{\psim}{\psi^{\sst-}}
\newcommand{\Wd}{w_2^{\sst-}}
\newcommand{\Wo}{w_1}
\newcommand{\f}{\varphi}
\newcommand{\fz}{\varphi_{\sst0}}
\newcommand{\gz}{g_{\sst0}}
\newcommand{\Fz}{F_{\sst0}}
\newcommand{\Jz}{J_{\sst0}}

\newcommand{\frg}{\mathfrak{g}}

\newcommand{\fre}{\mathfrak{e}}

\newcommand{\frsu}{\mathfrak{su}}

\newcommand{\U}{{\mathrm U}}
\newcommand{\SU}{{\mathrm{SU}}}

\newcommand{\SO}{{\mathrm {SO}}}

\newcommand{\GL}{{\mathrm {GL}}}

\newcommand{\G}{{\mathrm G}}
\newcommand{\K}{{\mathrm K}}

\newcommand{\SL}{{\mathrm {SL}}}

\newcommand{\ddt}{\frac{\partial}{\partial t}}

\newcommand{\W}{\wedge}

\newcommand{\Sym}{\mathrm{Sym}}
\DeclareMathOperator\tr{tr}

\DeclareMathOperator\End{End}

\DeclareMathOperator\vol{vol}

\newcommand{\Ric}{{\rm Ric}}

\newcommand{\Scal}{{\rm Scal}}

\newcommand{\st}{\ |\ }

\newcommand{\diag}{{\rm diag}}

\newcommand{\sst}{\scriptscriptstyle}
\newcommand{\td}{~}
\newcommand{\te}{\tilde{e}}

\numberwithin{equation}{section}

\title{Special solutions to the type IIA flow}

\author{Alberto Raffero}
\address{Dipartimento di Matematica ``G. Peano'' \\ Universit\`a degli Studi di Torino\\ Via Carlo Alberto 10\\10123 Torino\\ Italy}
\email{alberto.raffero@unito.it}

\subjclass[2020]{53C15, 53E50, 53D05, 53C25}
\keywords{Type IIA geometry; symplectic half-flat SU(3)-structure; Hermitian Ricci tensor}

\begin{document}
\begin{abstract}
We consider the source-free Type IIA flow introduced by Fei-Phong-Picard-Zhang \cite{FPPZ}, and we study it 
in the case where the relevant geometric datum is a symplectic half-flat SU(3)-structure. 
We show the existence of ancient, immortal and eternal solutions to the flow, provided that the initial symplectic half-flat structure satisfies suitable properties. 
In particular, we prove that the solution starting at a symplectic half-flat structure with Hermitian Ricci tensor is ancient and evolves self-similarly by scaling the initial datum. 
These results apply to all known (locally) homogeneous spaces admitting invariant symplectic half-flat SU(3)-structures.
\end{abstract}
\maketitle


\section{Introduction}
In recent years, various geometric flows have been introduced as potential tools to solve the non-linear partial differential equations 
related to supersymmetric compactifications of string theories \cite{FPPZ1,FPPZ,FPPZ2,Pho,PPZ,PPZ2}. 
In this paper, we focus on the {\em Type IIA flow} introduced by Fei-Phong-Picard-Zhang in \cite{FPPZ} to study the system of equations for the Type IIA string 
considered in \cite{TsYa} and motivated by \cite{GMPT,Tom}.

Let $(M,\omega)$ be a compact symplectic 6-manifold and let $\f$ be a 3-form on $M$ that is primitive with respect to $\omega$ and non-degenerate according to the definition of \cite{Hit}. 
Then, $\f$ gives rise to an almost complex structure $J_\f$ on $M.$ 
Moreover, the symplectic form $\omega$ is of type $(1,1)$ with respect to $J_\f$, and thus the tensor $g_\f\coloneqq \omega(\cdot,J_\f\cdot)$ is symmetric. 
If $g_\f$ is positive definite, $\f$ is said to be {\em positive}. 
According to \cite{FPPZ}, the pair $(\omega,\f)$ is called a {\em Type IIA geometry} if the  3-form $\f$ is closed, primitive and positive. 

A Type IIA geometry $(\omega,\f)$ gives rise to an almost K\"ahler structure $(\omega,J_\f,g_\f)$ and to a nowhere vanishing complex volume form $\Phi \coloneqq \f + i \, J_\f\f$ on $M.$ 
In particular, $M$ is a symplectic 6-manifold with trivial canonical bundle, and thus a {\em symplectic Calabi-Yau} manifold according to \cite{FiPa1,FiPa2}. 

A solution to the {\em Type IIA equation} is a Type IIA geometry $(\omega,\f)$ solving the non-linear partial differential equation
\begin{equation}\label{TypeIIAEquation}
dJ_\f \, d^* \left(|\f|^2\,\f \right) = \rho_A,  
\end{equation}
where $|\f|$ denotes the norm of $\f$ with respect to $g_\f$, $d^*$ is the formal adjoint of $d$, and the {\em source term} 
$\rho_A$ is the Poincar\'e dual of a given finite linear combination of special Lagrangians calibrated by $\f$. 
Notice that \eqref{TypeIIAEquation} is slightly different from the original equation considered in \cite{FPPZ,TsYa}, but they are equivalent as soon as 
the pair $(\omega,\f)$ defines a Type IIA geometry, see \cite[Sect.~3.2]{FPPZ}. 

In order to study the equation \eqref{TypeIIAEquation}, in \cite{FPPZ} the authors introduced a geometric flow evolving the 3-form $\fz$ of a given Type IIA geometry $(\omega,\fz)$ as follows
\begin{equation}\label{TypeIIAFlowIntro}
\begin{cases}
\frac{\partial}{\partial t} \f(t) =  dJ_{\f(t)}d^{*_t}\left(|\f(t)|_{g(t)}^2\f(t) \right) - \rho_A(t),  \\
\f(0) = \f_{\sst0}.
\end{cases}
\end{equation}
This flow is called the {\em Type IIA flow} and its stationary points are solutions to the Type IIA equation \eqref{TypeIIAEquation}.  

Assuming the source term $\rho_A(t)$ in \eqref{TypeIIAFlowIntro} to be identically zero, one obtains the so-called {\em source-free Type IIA flow}. 
By \cite{FPPZ}, this flow is well-posed on compact manifolds, and the solution $\f(t)$ remains closed, primitive and positive as long as it exists. 
Moreover, its stationary points are given by Type IIA geometries $(\omega,\f)$ such that the corresponding almost Hermitian structure $(\omega,J_\f,g_\f)$ is K\"ahler Ricci-flat. 
We refer the reader to \cite{FPPZ,FPPZ2} for further results on Type IIA geometry and on the Type IIA flow. 

A Type IIA geometry $(\omega,\f)$ gives rise to a symplectic SU(3)-structure which is defined by the almost K\"ahler structure $(\omega,J_\f,g_\f)$ together with 
the complex volume form of constant norm $\Psi = \frac{2}{|\f|}\,\Phi$. 
In particular, the whole structure is determined by $\omega$ and by the real 3-form $\psip\coloneqq \mathrm{Re}(\Psi)$, as the latter is primitive, positive and induces the same almost complex structure as $\f$. 
Using known results on SU(3)-structures \cite{BeVe,ChSa}, it is then possible to characterize certain properties of the Type IIA geometry $(\omega,\f)$, see Proposition \ref{TypeIIATypes}. 

When $\f$ has constant norm, the 3-form $\psip$ is closed, and the SU(3)-structure $(\omega,\psip)$ is {\em symplectic half-flat} (SHF for short). 
Known examples of compact 6-manifolds admitting SHF structures include the 6-torus  \cite{DeBTom,ToVe} and certain nilmanifolds and solvmanifolds \cite{CoTo,DeBTom0,FMOU, ToVe}, 
namely compact quotients of simply connected nilpotent or solvable Lie groups by cocompact discrete subgroups.  
In this last case, the examples are locally homogeneous. 
Further non-compact homogeneous and cohomogeneity one examples are given in \cite{PoRa1,PoRa2}. 
More details can be found in Section \ref{SpecialSHFSection}. 

In \cite{FiRa}, we observed that on all nilmanifolds and solvmanifolds admitting SHF structures there exists a SHF structure $(\omega,\psip)$ satisfying some distinguished properties. 
In detail, its {\em torsion form} $\Wd \coloneqq d^*\psip$ satisfies the conditions
\begin{equation}\label{SpecialSHFIntro}
\Delta_g\Wd = c\,\Wd,\qquad d\Wd\W\Wd = 0,\qquad |d\Wd|^2 = c\,|\Wd|^2,
\end{equation}
where $\Delta_g = dd^*+d^*d$ is the Hodge Laplacian of $g$, and $c$ is a real constant. 
Here, we shall call a SHF structure satisfying \eqref{SpecialSHFIntro} {\em special}. 
In Section \ref{SpecialSHFSection}, we discuss the conditions \eqref{SpecialSHFIntro} in detail. 
In particular, we show that the third condition follows from the first one under suitable assumptions on $(M,\omega,\psip)$, see Proposition \ref{specialrelations}. 
Moreover, in Proposition \ref{boundc} we prove that $c\geq \frac14 |\Wd|^2$ and that it attains the minimum value if and only if the SHF structure has Hermitian Ricci tensor. 
A few examples of SHF structures satisfying this remarkable curvature constraint are known, see \cite{PoRa1}. 

In Proposition \ref{autexact}, we show that a compact 6-manifold with a SHF structure $(\omega,\psip)$ has finite automorphism group whenever the 3-form $\psip$ is exact. 
This gives new information on the automorphism group of a compact SHF manifold in addition to the results obtained in \cite{PoRa2}. 
Moreover, it holds for SHF structures with Hermitian Ricci tensor, and suggests that symmetry techniques might not be helpful in the search of new compact examples. 

Some consequences of the second and third condition in \eqref{SpecialSHFIntro} are discussed in Section \ref{LemmaSect}. 

In \cite{FiRa}, we were interested in studying a coupled flow for SU(3)-structures related to Bryant's G$_2$-Laplacian flow \cite{Bry}. 
The conditions \eqref{SpecialSHFIntro} turned out to be useful to describe the solution to the flow starting at a SHF structure satisfying them. 
Motivated by this result, in Section \ref{FlowSect} we investigate the source-free Type IIA flow starting at a Type IIA geometry $(\omega,\fz)$ whose corresponding 
SU(3)-structure $(\omega,\psip_{\sst0})$ is SHF and special.  

We first consider the case where the Ricci tensor of the metric $g_{\sst0}$ induced by $(\omega,\fz)$ is Hermitian with respect to the almost complex structure $J_{\f_0}$, 
and we show that the solution to the flow starting at $\fz$  is ancient and it evolves self-similarly by scaling the initial datum $\fz$ (see Theorem \ref{JHermFlow}).  

We then focus on the general case where $(\omega,\psip_{\sst0})$ is SHF and special. 
Assuming $(M,\omega,\fz)$ to be locally homogeneous, so that the norm of the torsion form $\Wd = d^{*_0}\psip_{\sst0}$ is constant, and that a suitable property holds, 
in Theorem \ref{SolThm} we show that a solution to the flow starting at $\fz$ is given by the 3-form 
\[
\f(t) = \fz + \frac{a(t)}{|\fz|_{\sst0}}\, \Delta_{g_0}\fz, 
\]
where $|\fz|_{\sst0}$ is the (constant) $g_{\sst0}$-norm   of $\fz$, and $a(t)$ is a real valued function solving the initial value problem
\[
\left\{
\begin{split}
\frac{d}{dt} a(t) &= |\fz|_{\sst0}^3\left(1+\frac{c}{|\fz|_{\sst0}} \, a(t) \right)^{\mbox{ $\frac{|\Wd|_{\sst0}^2}{c}-1$ }}, \\
a(0) &= 0.
\end{split} \right.
\]
Notice that no compactness assumption is needed to show this result, but it is required if one wants to conclude that the solution described above is the unique one starting at $\fz$.
 
Depending on the relation between $c$ and $|\Wd|_{\sst0}^2$, we get different types of solutions that exist on a maximal time interval of the form   
$(-\infty,+\infty)$, $(-\infty,T)$ or $(T,+\infty)$, for a certain $T<\infty$, and that are thus {\em eternal}, {\em ancient} or {\em immortal}, respectively 
(see corollaries \ref{CorEternal}, \ref{CorAncient}, \ref{CorImmortal}). 

These results apply to all known homogeneous and compact locally homogeneous manifolds admitting invariant SHF structures.  
Moreover, they allow us to obtain further examples of solutions to the Type IIA flow in addition to those determined in \cite[Sect. 9.3.2]{FPPZ} 
on the nilmanifold example given in \cite[Ex.~5.2]{DeBTom} and on the solvmanifold example given in \cite[Thm.~3.5]{ToVe}.


\section{Definite 3-forms and SU(3)-structures in six dimensions}\label{preliminaries}

\subsection{Definite 3-forms}
We briefly recall the definition of a definite 3-form on a six-dimensional vector space and how it is possible to construct a complex structure out of it. 
For more details, we refer the reader to \cite{Hit,Hit1,SaKi}. 

A 3-form $\f$ on a six-dimensional real vector space $V$ is said to be {\em definite}, or {\em non-degenerate}, if the contraction $\iota_v\f\in\Lambda^2V^*$ has rank $4$, for every non-zero vector $v\in V$. 
The set of all definite 3-forms on $V$ coincides with one of the two open orbits for the natural action of $\GL(V)$ on $\Lambda^3V^*.$ 
In detail, if $A:\Lambda^5V^*\rightarrow V\otimes \Lambda^6V^*$ denotes the isomorphism induced by the wedge product $\W : \Lambda^5V^*\otimes V^* \rightarrow \Lambda^6V^*$, and 
$\Omega \in \Lambda^6V^*$ is a non-zero element, then a 3-form $\f\in\Lambda^3V^*$ gives rise to an endomorphism $S_\f: V \rightarrow V$ defined as follows
\[
A(\iota_v\f\W\f) = S_\f(v)\,\Omega,
\]
for every $v\in V.$ 
The endomorphism $S_\f$ satisfies the identity $S_\f^2 = P(\f)\mathrm{Id}_V$, where $P(\f)$ is an irreducible $\SL(V)$-invariant  polynomial of degree $4$. 
The $\GL(V)$-orbit of definite 3-forms is then the open subset of $\Lambda^3V^*$
\[
\Lambda^3_{\sst-} V^* \coloneqq \{\f\in\Lambda^3V^* \st P(\f)<0\}. 
\]
Every definite form $\f \in \Lambda^3_{\sst-}V^*$ has $\GL(V)$-stabilizer conjugate to $\SL(3,\C)$, and thus it defines a complex structure  
\[
J_\f :V\rightarrow V,\quad J_\f = \frac{1}{\sqrt{-P(\f)}} \, S_\f. 
\]
The 3-form $J_\f\f = \f(J_\f\cdot,J_\f\cdot,J_\f\cdot)$ is also definite and it induces the same complex structure $J_\f$ as $\f$. Moreover, the complex 3-form  $\f + iJ_\f\f$ is of type $(3,0)$ with respect to $J_\f$.  
More generally, the 3-form $\rho\, \mathrm{Re}\left(e^{-i\theta} \left(\f + iJ_\f\f \right) \right)$ is definite and induces the complex structure $J_\f$,   
for every $\rho\,e^{-i\theta}\in\C\smallsetminus\{0\}$. 

\subsection{SU(3)-structures} 
An SU(3)-structure on a six-dimensional vector space $V$ is the data of an Hermitian structure $(g,J)$ with fundamental 2-form $\omega=g(J\cdot,\cdot)$ and a complex $(3,0)$-form $\Psi$ of constant norm.  

According to \cite{Hit1}, SU(3)-structures can be characterized in terms of non-degenerate forms as follows. 
Let $\omega\in\Lambda^2V^*$ be a non-degenerate 2-form on $V.$ 
Consider a definite 3-form $\psip\in\Lambda^3_{\sst-}V^*$, denote by $J$ the complex structure determined by $\psip$ and the orientation $\omega^3$, and denote by $\psim \coloneqq J\psip$ the imaginary part 
of the complex $(3,0)$-form $\Psi = \psip + i J\psip$.  
Then, the pair $(\omega,\psip)$ defines an SU(3)-structure on $V$ if and only if the following conditions are satisfied
\begin{enumerate}[-]
\item $\psip$ is {\em primitive} with respect to $\omega$, i.e., $\psip\W\omega=0$. This is equivalent to $\omega$ being of type $(1,1)$ with respect to $J$, namely $\omega(J\cdot,J\cdot)=\omega$;
\item the symmetric bilinear form $g\coloneqq \omega(\cdot,J\cdot)$ is positive definite;
\item the following {\em normalization condition} holds 
\begin{equation}\label{normalization}
\psip \W \psim = \frac23 \, \omega^3 = 4 \vol_g,  
\end{equation}
where $\vol_g$ denotes the volume form of the metric $g$. 
\end{enumerate}

Given an SU(3)-structure $(\omega,\psip)$ on $V$ with corresponding complex structure $J$ and metric $g$, 
there exists a $g$-orthonormal basis $\mathcal{B}=(e_1,\ldots,e_6)$ of $V$ with dual basis $\mathcal{B}^* = (e^1,\ldots,e^6)$ such that 
\begin{equation}\label{adaptedfr}
\omega = e^{12}+e^{34}+e^{56},\quad
\psip	= e^{135}-e^{146}-e^{236}-e^{245},\quad
\psim = e^{136}+e^{145}+e^{235}-e^{246}, 
\end{equation}
and $J(e_{2k-1}) = e_{2k}$, for $k=1,2,3$, where the symbol $e^{ij\cdots}$ is a shortening for the wedge product of covectors $e^i\W e^j \W \cdots$. 
The bases $\mathcal{B}$ and $\mathcal{B}^*$ are said to be {\em adapted} to the SU(3)-structure $(\omega,\psip)$. 

Let $*_g$ denote the Hodge operator induced by $g$ and the orientation on $V.$ 
Working with respect to an adapted basis, it is immediate to check that $\psim = *_g\psip$ and $*_g\omega=\tfrac12 \omega^2$.   

\smallskip

The group SU(3) acts irreducibly on the spaces $V^*$ and $\Lambda^5V^*$, 
while the spaces $\Lambda^2V^*$ and $\Lambda^3V^*$ decompose into $g$-orthogonal irreducible SU(3)-modules as follows (see e.g.~\cite{BeVe,ChSa})
\[
\begin{split}
\Lambda^2V^*	&= \Lambda^2_1V^*\oplus \Lambda^2_6V^*\oplus \Lambda^2_8V^*,\\
\Lambda^3V^*	&= \Lambda^3_{\mathrm{Re}}V^* \oplus \Lambda^3_{\mathrm{Im}}V^* \oplus \Lambda^3_6V^*\oplus \Lambda^3_{12}V^*, 
\end{split}
\]
where the submodules $\Lambda^3_{\mathrm{Re}}V^*$ and $\Lambda^3_{\mathrm{Im}}V^*$ are one-dimensional, and $\Lambda^k_nV^*$ denotes an $n$-dimensional irreducible submodule of $\Lambda^kV^*$.  
In detail
\begin{equation}\label{lambda2}
\begin{split}
\Lambda^2_1V^*	&= \R\,\omega,\\
\Lambda^2_6V^*	&= \left\{*_g(\alpha\W\psip) \st \alpha\in V^* \right\}, \\ 
\Lambda^2_8V^*	&= \left\{\sigma\in\Lambda^2V^* \st J\sigma=\sigma \mbox{ and } \sigma\W\omega^2=0\right\} \cong \mathfrak{su}(3),  
\end{split}
\end{equation}
and 
\begin{equation}\label{lambda3}
\begin{split}
\Lambda^3_{\mathrm{Re}}V^*	&= \R\,\psip, \quad  \Lambda^3_{\mathrm{Im}}V^*	= \R\,\psim, \\
\Lambda^3_6V^*	&= \left\{\alpha \W \omega \st \alpha\in V^*\right\}, \\
\Lambda^3_{12}V^*	&= \left\{\rho\in\Lambda^3V^* \st \rho\W\omega=0 \mbox{ and } \rho\W\psi^{\sst\pm}=0\right\}. 
\end{split}
\end{equation}
Moreover, the decomposition of the space $\Lambda^4V^*$ follows from that of $\Lambda^2V^*$ by applying the Hodge operator $*_g$. 

Notice that every $\sigma\in\Lambda^2_8V^*$ satisfies the identity $*_g(\sigma\W\omega)=-\sigma$, and that the Hodge dual of every $\rho\in\Lambda^3_{12}V^*$ is given by $*_g\rho = J\rho$ 
(see \cite[Remark 2.7]{BeVe}).


\section{From Type IIA geometry to symplectic SU(3)-structures}\label{SectTypeIIASympl}
We now review the definition of a Type IIA geometry introduced in \cite{FPPZ}, and we discuss how this type of geometric structure is related to symplectic SU(3)-structures. \smallskip

Let $(M,\omega)$ be a six-dimensional symplectic manifold and let $\f\in\Omega^3(M)$ be a {\em definite} 3-form, namely $\f|_x\in\Lambda^3_{\sst-}(T^*_xM)$ is definite at each point $x$ of $M.$ 
We denote by $J_\f\in\End(TM)$ the almost complex structure induced by $\f$ and the orientation $\omega^3$. 
If $\f$ is primitive with respect to $\omega$, then $\omega$ is of type $(1,1)$ with respect to $J_\f$ and, consequently, 
the 2-covariant tensor $g_\f \coloneqq \omega(\cdot,J_\f\cdot)$ is symmetric. 
In such a case, $\f$ is said to be {\em positive} if $g_\f$ defines a Riemannian metric.   
According to \cite{FPPZ}, the pair $(\omega,\f)$ is called a {\em Type IIA geometry} if $\f$ is closed, primitive and positive. \smallskip

Let $(\omega,\f)$ be a Type IIA geometry. Since $\f$ is definite, its norm $|\f| = g_\f(\f,\f)^{1/2}$ is a nowhere vanishing function, and we can consider the complex $(3,0)$-form  
\[
\Psi \coloneqq \frac{2}{|\f|}\,\Phi = \frac{2}{|\f|} \, \f + i \, \frac{2}{|\f|} \, J_\f\f. 
\]  
By \cite{FPPZ}, $J_\f\f=*_{g_\f}\f$, where $*_{g_\f}$ denotes the Hodge operator determined by $g_\f$ and the given orientation on $M.$ 
Therefore, the real and imaginary parts of $\Psi$ 
\begin{equation}\label{psippsimIIA}
\psip = \frac{2}{|\f|} \, \f,\quad \psim =  \frac{2}{|\f|} \, J_\f\f, 
\end{equation}
satisfy the normalization condition \eqref{normalization} everywhere on $M.$ Consequently, the pair $(\omega,\psip)$ defines an $\SU(3)$-structure on $M.$  

Since $\psip$ is proportional to $\f$, it is primitive and positive, and it induces the same almost complex structure as $\f$, for any given orientation on $M.$  
Therefore, the Riemannian metric $g_{\psip}$ associated with the pair $(\omega,\psip)$ as explained before coincides with the metric $g_\f$ corresponding to the Type IIA geometry $(\omega,\f)$. 
From now on, we shall refer to the pair $(\omega,\psip)$ as the {\em symplectic $\SU(3)$-structure} corresponding to the Type IIA geometry $(\omega,\f)$, 
and we shall denote the almost complex structure and Riemannian metric associated with these structures by $J$ and $g$, respectively. 

\smallskip

By \cite{ChSa}, the intrinsic torsion of an SU(3)-structure $(\omega,\psip)$ is determined by $d\omega$, $d\psip$ and $d\psim$. 
In detail, one can decompose these forms according to the $g$-orthogonal decomposition of the bundle $\Lambda^kT^*M$ that is induced by the irreducible decomposition 
of the $\SU(3)$-module $\Lambda^kV^*$, for $k=3,4$ (see Section \ref{preliminaries}). 
Each summand corresponds then to a component of the intrinsic torsion in the $\SU(3)$-irreducible splitting of $V^*\otimes\frsu(3)^\perp$. 
Using the notation of \cite{BeVe}, the decompositions of the spaces of $3$- and $4$-forms on $M$ can be described as follows
\begin{equation}\label{FormsSplit3}
\begin{split}
\Omega^3(M) &= \Omega^3_{\mathrm{Re}}(M) \oplus   \Omega^3_{\mathrm{Im}}(M)  \oplus \Omega^3_6(M)  \oplus \Omega^3_{12}(M), \\
\Omega^4(M) &= \Omega^4_1(M) \oplus   \Omega^4_6(M)  \oplus \Omega^4_8(M),
\end{split}
\end{equation}
where the definition of each summand follows from \eqref{lambda2}, \eqref{lambda3} and the correspondence $\Lambda^4_nV^* = *_g\Lambda^2_nV^*$. 

\smallskip

As for the symplectic $\SU(3)$-structure $(\omega,\psip)$ corresponding to a Type IIA geometry $(\omega,\f)$, we have $d\omega=0$, 
and a direct computation using the expression \eqref{psippsimIIA} of $\psip$ shows that
\[
d\psip = \Wo\W\psip \in \Omega^4_6(M),  
\]
where $\Wo \coloneqq d \log \frac{1}{|\f|}$. 
Now, from the results of \cite{BeVe}, we deduce that $d\psim \in  \Omega^4_6(M)  \oplus \Omega^4_8(M)$, and that its $\Omega^4_6(M)$-component  is $J\Wo\W\psip$. 
Since $J\f = *_g\f$, we then obtain
\[
d\psim = d\log\frac{1}{|\f|} \W \psim + \frac{2}{|\f|} \, d*_g\f = J\Wo \W\psip + \Wd \W \omega,
\]
where we used the identity $\alpha\W\psim = J\alpha\W\psip$, which holds for every 1-form $\alpha$ on $M,$ 
and we let 
\[
\Wd \coloneqq \frac{2}{|\f|}\,d^*\f = - \frac{2}{|\f|}*_g d \,*_g \f  \in \Omega^2_8(M).
\] 
Therefore, the forms $\Wo$ and $\Wd$ are the {\em torsion forms} of the symplectic SU(3)-structure $(\omega,\psip)$ corresponding to a Type IIA geometry (cfr.~\cite[Def.~2.10]{BeVe}). 

\smallskip

The next proposition summarizes how certain properties of $(\omega,\psip)$ - and, thus, of $(\omega,\f)$ - are related to $\Wo$ and $\Wd$. 
The proof immediately follows from the results of \cite{ChSa} and the above expressions of the torsion forms.

\begin{proposition}\label{TypeIIATypes}
Let $(\omega,\f)$ be a Type IIA geometry with corresponding symplectic $\SU(3)$-structure $(\omega,\psip)$. Then 
\begin{enumerate}[$i)$]
\item\label{iTypes} $J$ is integrable if and only if $\Wd = 0$, and thus if and only if $d^*\f =0$;
\item $d\psip = 0$ if and only if $\Wo = 0$, and thus if and only if $|\f|$ is constant;  
\item\label{iiiTypes} the holonomy group $\mathrm{Hol}(g)$ of the Riemannian metric $g$ is a subgroup of ${\SU}(3)$ if and only if $\Wo=0$ and $\Wd=0$, 
and thus if and only if $|\f|$ is constant and $J$ is integrable. 
\end{enumerate}
\end{proposition}

Notice that the vanishing of the torsion forms $\Wo$ and $\Wd$ is equivalent to the vanishing of the intrinsic torsion of the SU(3)-structure $(\omega,\psip)$. 
In such a case, $(\omega,\psip)$ is called {\em torsion-free} and the corresponding Riemannian metric $g$ is Ricci-flat, i.e., $\Ric_g=0$. 
In particular, the almost Hermitian structure $(\omega,J,g)$ is K\"ahler Ricci-flat. 

\begin{remark}
Points \ref{iTypes}) and \ref{iiiTypes}) of Proposition \ref{TypeIIATypes} can also be shown without introducing the symplectic SU(3)-structure associated with $(\omega,\f)$, see \cite{FPPZ}. 
\end{remark} 
 
\smallskip

 
\section{Special symplectic half-flat $\SU(3)$-structures}\label{SpecialSHFSection}
We now focus on the case where the Type IIA geometry $(\omega,\f)$ is defined by a 3-form $\f$ of constant norm. This happens, for instance, when $(M,\omega,\f)$ is a (locally) homogeneous space. 
Under this assumption, we have $\Wo=0$, and thus the SU(3)-structure $(\omega,\psip)$ corresponding to $(\omega,\f)$ satisfies the conditions 
\[
d\omega = 0,\quad d\psip	= 0, \quad d\psim = \Wd \W \omega, 
\]
where $\Wd = d^*\left(\frac{2}{|\f|}\f\right) = d^*\psip \in \Omega^2_8(M)$. 
In the literature, these SU(3)-structures are known as  {\em symplectic half-flat}  ({\em SHF} for short) \cite{CoTo,PoRa1,PoRa2, ToVe} or {\em special generalized Calabi-Yau} \cite{DeBTom0,DeBTom}. 
According to the classification by Chiossi-Salamon \cite{ChSa}, the symplectic SU(3)-structures we are considering constitute the class $\mathcal{W}_2^-\oplus\mathcal{W}_5$, 
while SHF structures determine the subclass $\mathcal{W}_2^-$.

\smallskip

From \cite{ChSa}, we know that the almost complex structure $J$ induced by a SHF structure $(\omega,\psip)$ is integrable if and only if $\Wd$ vanishes. 
Consequently, the torsion form $\Wd$ determines the Nijenhuis tensor of $J$ 
\[
N_J(X,Y) = \frac14 \left([JX,JY]  - J [JX,Y] -J[X,JY] - [X,Y] \right).
\]
In particular, their norms are related as follows
\begin{equation}\label{NJSHF}
|N_J|^2 = \frac12\, |\Wd|^2.
\end{equation}
Moreover, by \cite[Thm.~3.4]{BeVe}, the scalar curvature of the metric $g$ corresponding to $(\omega,\psip)$ is given by 
\[
\Scal_g = -\frac12\, |\Wd|^2,
\]
and thus it vanishes identically if and only if $(\omega,\psip)$ is torsion-free. 

We summarize some further properties of the torsion form $\Wd$ in the next lemma. For a proof, we refer the reader to \cite[Lemma 5.1]{FiRa}. 
\begin{lemma}\label{dwdlemma}
Let $(\omega,\psip)$ be a symplectic half-flat $\SU(3)$-structure. Then, the torsion form $\Wd = d^*\psip$ is coclosed, 
and its exterior differential decomposes as follows with respect to the decomposition \eqref{FormsSplit3} of $\Omega^3(M)$
\begin{equation}\label{dWdexpr}
d\Wd = \frac{\left|\Wd\right|^2}{4}\,\psip +\gamma,
\end{equation}
for a unique $\gamma\in\Omega^3_{12}(M)$. 

\noindent As a consequence, the following identities hold
\begin{equation}\label{iddw2}
\begin{array}{rcl} 
d\Wd \wedge \psip  &=& *_g d\Wd \wedge \psim  = 0, \vspace{0.1cm} \\ 
d \Wd \wedge \psim  &=&  \psip \W*_g d\Wd =  \left| \Wd \right|^2 \vol_g.
\end{array}
\end{equation}
Moreover, if $\gamma$ vanishes identically, then $|\Wd|$ is constant.  
\end{lemma}

To the best of our knowledge, the currently known examples of compact 6-manifolds admitting SHF structures are given by the 6-torus $\mathbb{T}^6$ (see \cite{DeBTom,ToVe}) and by 
certain nilmanifolds and solvmanifolds (see \cite{CoTo,DeBTom0,DeBTom,FMOU, ToVe}). In this last case, the 6-manifold $M$ is the compact quotient of 
a six-dimensional simply connected nilpotent or solvable Lie group $\G$ by a cocompact discrete subgroup (lattice) $\Gamma\subset\G$, 
and the SHF structure $(\omega,\psip)$ on $M=\Gamma\backslash\G$ is induced by a left-invariant one on $\G$. 
In particular, $(\Gamma\backslash\G, \omega,\psip)$ is a locally homogeneous space that is not globally homogeneous. 
Indeed, the SHF structure $(\omega,\psip)$ is preserved by the local diffeomorphisms of $\Gamma\backslash\G$ induced by the locally defined left translation action of $\G$ on $\Gamma\backslash\G$, 
while it is not preserved by the transitive right translation action of $\G$ on $\Gamma\backslash\G$. 
Notice that examples of this form may occur only when $\G$ is unimodular and solvable. 
Indeed, a Lie group admits lattices only if it is unimodular \cite{Mil}, and a unimodular Lie group admitting left-invariant symplectic structures must be solvable \cite{Chu}. 

A left-invariant SHF structure on a Lie group $\G$ is determined by an SU(3)-structure of the same type on the Lie algebra $\frg$ of $\G$ and, conversely, a SHF structure on a Lie algebra $\frg$ 
gives rise to a left-invariant SHF structure on every Lie group corresponding to $\frg$.  
The classification of the isomorphism classes of nilpotent Lie algebras admitting SHF structures was given in \cite{CoTo}, while the analogous classification in the case of solvable 
Lie algebras was obtained in \cite{FMOU}. 
Before summarizing these results in the next theorem, we recall a useful notation. 
Given a six-dimensional Lie algebra $\frg$, its structure equations with respect to a basis of covectors $\left(e^1,\ldots,e^6 \right)$ 
are described by the $6$-tuple $(d e^1,\ldots, d e^6)$, where $d$ denotes the Chevalley-Eilenberg differential of $\frg$. 

\begin{theorem}[\cite{CoTo,FMOU}]\label{solvableSHF}
Let $\frg$ be a six-dimensional, unimodular, non-abelian, solvable Lie algebra. Then, $\frg$ admits symplectic half-flat $\SU(3)$-structures if and only if 
it is isomorphic to one of the following Lie algebras
\begin{eqnarray*}
\fre(1,1)\oplus\fre(1,1)			&=&	(0,-e^{13},-e^{12},0,-e^{46},-e^{45});\\
\frg_{5,1}\oplus\R				&=& (0,0,0,0,e^{12},e^{13});\\
A_{5,7}^{-1,-1,1}\oplus\R			&=& (e^{15},-e^{25},-e^{35},e^{45},0,0);\\
A_{5,17}^{a,-a,1}\oplus\R			&=& (a e^{15}+e^{25},-e^{15}+a e^{25},-a e^{35}+e^{45},-e^{35}-a e^{45},0,0),~a>0;\\
\frg_{6,N3}					&=& (0,0,0,e^{12},e^{13},e^{23});\\
\frg_{6,38}^{0}					&=& (e^{23},-e^{36},e^{26},e^{26}-e^{56},e^{36}+e^{46},0);\\
\frg_{6,54}^{0,-1}				&=& (e^{16} + e^{35}, -e^{26} + e^{45}, e^{36}, -e^{46}, 0, 0);\\
\frg_{6,118}^{0,-1,-1}				&=& (-e^{16} +e^{25},-e^{15} -e^{26}, e^{36} -e^{45}, e^{35} +e^{46}, 0, 0).
\end{eqnarray*}
The nilpotent Lie algebras appearing in the previous list are $\frg_{5,1}\oplus\R$ and $\frg_{6,N3}$. 
\end{theorem}

A 6-manifold $M$ with an SU(3)-structure $(\omega,\psip)$ is said to be {\em homogeneous} if it is acted on transitively by its automorphism group 
\[
\mathrm{Aut}(M,\omega,\psip) = \left\{f\in\mathrm{Dif{}f}(M) ~|~ f^*\omega=\omega,~f^*\psip=\psip \right\},
\] 
or a subgroup thereof. 
By  \cite[Thm.~2.1]{PoRa2}, if $M$ is compact and $(\omega,\psip)$ is a SHF structure that is not torsion-free, 
then the identity component of $\mathrm{Aut}(M,\omega,\psip)$ is a $k$-torus with $k\leq \mathrm{min}\{5,b_1(M)\}$. 
Consequently, a homogeneous SHF 6-manifold cannot be compact unless it is a flat torus with a torsion-free $\SU(3)$-structure. 

The homogeneous examples given by six-dimensional Lie groups $\G$ with a left-invariant SHF structure arise in the case where  
$\G \subseteq \mathrm{Aut}(M,\omega,\psip)$ acts simply transitively on $M.$ 
More generally, if the transitive $\G$-action is not free, 
then $M$ is $\G$-equivariantly diffeomorphic to the quotient $\G/\K$, where $\K$ is a compact subgroup of $\G$, 
and the SHF-structure $(\omega,\psip)$ on $M=\G/\K$ is $\G$-invariant. If $\G$ is semisimple, the following classification result holds. 

\begin{theorem}[\cite{PoRa1}]\label{HomSHF}
Let $(M,\omega,\psip)$ be a symplectic half-flat $6$-manifold which is homogeneous under the action of a semisimple Lie group $\G$. 
Then, $M$ is non-compact and one of the following situations occurs
\begin{enumerate}[-]
\item $M = \SU(2,1)/\mathrm{T}^2$, and there exists a $1$-parameter family of pairwise non-homothetic and non-isomorphic invariant SHF structures; 
\item $M = \SO_{\sst0}(4,1)/\U(2)$, and there exists a unique invariant SHF structure up to homothety. 
\end{enumerate}
In both cases, the Riemannian metric $g$ induced by the SHF structure has Hermitian Ricci tensor, namely $\Ric_g(J\cdot,J\cdot) = \Ric_g$. 
\end{theorem}

Requiring the Ricci tensor to be Hermitian is a meaningful curvature constraint for SHF structures. 
Indeed, a SHF structure whose Riemannian metric $g$ is Einstein must be torsion-free \cite[Cor.~4.1]{BeVe}, 
and the Hermitian condition is a natural replacement of the Einstein condition on almost K\"ahler manifolds, see  \cite{BlIa}. 

Using the description of the Ricci tensor of an SU(3)-structure given in \cite{BeVe}, 
in \cite{PoRa1} we proved that SHF structures with Hermitian Ricci tensor can be characterized as follows.  

\begin{proposition}[\cite{PoRa1}]\label{JRic}
A symplectic half-flat $\SU(3)$-structure $(\omega,\psip)$ has Hermitian Ricci tensor if and only if 
\[
\Delta_g\psip = \lambda\,\psip.
\]   
When this happens, $\lambda = \frac{1}{4}\,|\Wd|^2$ and it is constant. Consequently, $\Scal_g$ is constant, too. 
\end{proposition}

If $(\omega,\psip)$ is a SHF structure with Hermitian Ricci tensor, then the previous result implies that $\psip$ is an exact 3-form, as $\Delta_g\psip = dd^*\psip = d\Wd$. 
The converse is not true, as there exists an example of SHF structure $(\omega,\psip)$ on $\mathbb{S}^3 \times \R^3$ such that  $\psip$ is exact and $\Ric_g$ is not Hermitian, see \cite{PoRa2}. 
This example is complete and invariant under the natural cohomogeneity one  action of $\SO(4)$ on $\mathbb{S}^3 \times \R^3$. 

In addition to the non-compact homogeneous examples of Theorem \ref{HomSHF}, an example of SHF structure with Hermitian Ricci tensor occurs on the Lie algebra $\fre(1,1)\oplus\fre(1,1)$, 
and thus on compact quotients of the corresponding simply connected Lie group, see \cite[Thm.~3.5]{ToVe} and Proposition \ref{SpecialFR} below. 
We are not aware of the existence of further examples that are not (locally) homogeneous.  
In fact, requiring $\psip$ to be exact imposes strong constraints on the symmetries of a compact SHF manifold $(M,\omega,\psip)$, as the next result shows.

\begin{proposition}\label{autexact}
The automorphism group of a compact symplectic half-flat 6-manifold $(M,\omega,\psip)$ with $[\psip] = 0 \in {H}^3_{\mathrm{dR}}(M) $ is finite. 
\end{proposition}
\begin{proof}
The argument is analogous to the one used to show that exact $\G_2$-structures do not have non-trivial infinitesimal symmetries \cite{Fow}. 
Since $M$ is compact, the automorphism group $\mathrm{Aut}(M,\omega,\psip)$ is a compact Lie group with Lie algebra 
\[
\mathfrak{aut}(M,\omega,\psip) = \left\{X\in \Gamma(TM) \st \mathcal{L}_{X}\omega=0,~\mathcal{L}_{X}\psi=0 \right\}. 
\]
Consider an infinitesimal automorphism $X\in\mathfrak{aut}(M,\omega,\psip)$. Then, the forms $\iota_{X}\omega$ and $\iota_{X}\psip$ are closed.  
Now, using general identities involving the forms $\omega$ and $\psip$ (see e.g. \cite[Lemma 3.7]{FiRa}), we see that 
\[
\iota_X\omega\W\iota_X\psip\W\psip = -2 JX^\flat \W *_g(JX^\flat) = -2|X|^2 \vol_g. 
\]
Integrating the previous identity over $M$, using that $\psip = d\alpha$, for some $\alpha\in\Omega^2(M)$, and that the forms $\iota_{X}\omega$, $\iota_{X}\psip$ are closed, we obtain
\[
-2 \int_M |X|^2 \mbox{vol}_g = \int_M \iota_X\omega\W\iota_X\psip\W\psip  = \int_M d\left(\iota_X\omega\W\iota_X\psip\W\alpha \right).
\]
Using Stokes' Theorem, we then conclude that $X=0$. Consequently, $\mathfrak{aut}(M,\omega,\psip)  = \{0\}$, and the thesis follows. 
\end{proof}

We now introduce a class of SHF structures that generalizes the class of SHF structures with Hermitian Ricci tensor. 
\begin{definition}\label{specialSHF}
We say that a symplectic half-flat $\SU(3)$-structure $(\omega,\psip)$  is {\em special} if its torsion form $\Wd$ is non-vanishing and satisfies the following properties
\begin{enumerate}[i)]
\item\label{ispec} $\Delta_g \Wd = c\,\Wd$, for some real number $c$;\vspace{0.1cm}
\item\label{iispec} $d\Wd \W \Wd =0$;\vspace{0.1cm}
\item\label{iiispec} $|d\Wd|^2 = c\, |\Wd|^2$.
\end{enumerate}
\end{definition}

Before reviewing some known examples of special SHF structures, we discuss the properties \ref{ispec})--\ref{iiispec}) appearing in the definition. 
First, we observe that SHF structures with Hermitian Ricci tensor are always special with  $c = \frac{1}{4}\,|\Wd|^2$. 
\begin{proposition}\label{boundc}
Let $(\omega,\psip)$ be a symplectic half-flat $\SU(3)$-structure with Hermitian Ricci tensor. Then, $(\omega,\psip)$ is special, and its torsion form $\Wd$ 
satisfies the equation $\Delta_g \Wd = c\,\Wd$ for $c = \frac{1}{4}\,|\Wd|^2$.
\end{proposition}
\begin{proof}
Since $\Delta_g\psip = dd^*\psip = d\Wd$, from Proposition \ref{JRic} we see that  
\[
d\Wd = \frac{1}{4}\,|\Wd|^2\,\psip, 
\]
and that $|\Wd|$ is constant. 
Since $\Wd$ is coclosed, taking the codifferential of the previous identity gives $\Delta_g \Wd = c\,\Wd$, with $c = \frac{1}{4}\,|\Wd|^2$. 
Moreover, the identity $d\Wd \W \Wd = 0$ immediately follows, as $\Wd$ is of type $(1,1)$ with respect to $J.$ 
Finally, using Lemma \ref{dwdlemma}, we see that 
\[
|d\Wd|^2 \vol_g = d\Wd \W *_g d\Wd =  \frac{1}{4}\,|\Wd|^2\,\psip \W *_g d\Wd =  \frac{1}{4}\,|\Wd|^4  \vol_g, 
\]
and thus $|d\Wd|_g^2 = c\, |\Wd|_g^2$. 
\end{proof}

In view of \eqref{dWdexpr}, the condition \ref{iispec}) is satisfied if and only if $\gamma\W\Wd = 0$, where $\gamma$ is the $\Omega^3_{12}(M)$-component of $d\Wd$. 
Moreover, it implies the identity $\Wd \W *_gd\Wd = 0$, since $*_g\gamma = J\gamma$ and $J\Wd=\Wd$  (see Section \ref{preliminaries}). 
The condition \ref{iispec}) is related to a general property discussed in Section \ref{LemmaSect}, and it will play a role in the proof of Theorem \ref{SolThm}. 
However, it might not be satisfied by the torsion form of a generic SHF structure. 

\smallskip

Under certain assumptions, the condition \ref{iiispec}) follows from \ref{ispec}), as the next result shows. 
\begin{proposition}\label{specialrelations}
Let $(\omega,\psip)$ be a symplectic half-flat $\SU(3)$-structure on a 6-manifold $M,$ and assume that $\Delta_g \Wd = c\,\Wd$, for some real number $c$. 
Then, the identity $|d\Wd|^2 = c\, |\Wd|^2$ holds in the following cases:
\begin{enumerate}[1)]
\item $M$ is compact and both $|d\Wd|$ and $|\Wd|$ are constant. This holds, in particular, if $(M,\omega,\psip)$ is compact and locally homogeneous;
\item $M=\G$ is a unimodular Lie group, and $(\omega,\psip)$ is left-invariant. 
\end{enumerate}
\end{proposition}
\begin{proof}\ 
\begin{enumerate}[1)]
\item Let $(\alpha,\beta)_g = \int_M\alpha\W*_g\beta$ denote the $L^2$ inner product on $\Omega^k(M)$ induced by $g$. 
Then,  
\[
|d\Wd|^2\, \mathrm{Vol}(M) = \left( d\Wd,d\Wd\right)_g = \left( d^*d\Wd,\Wd\right)_g = \left( \Delta_g\Wd,\Wd\right)_g = c\,|\Wd|^2\,  \mathrm{Vol}(M). 
\]
\item We can work on the Lie algebra $\frg$ of $\G$ with the SHF structure $(\omega,\psip)$ induced by the left-invariant one on $\G$. 
Let $d$ denote the Chevalley-Eilenberg differential of $\frg$. Since the Lie group is unimodular, every $5$-form on $\frg$ is $d$-closed. Consequently, we see that 
\[
|d\Wd|^2 \vol_g = d\Wd \W *_g d\Wd = d \left( \Wd \W *_g d\Wd \right) - \Wd \W d *_g d\Wd = - \Wd \W d *_g d\Wd. 
\]
Now, we have $c\,\Wd = \Delta_g \Wd = -*_g d\,*_g d \Wd$, whence it follows that $c\,*_g\Wd =  - d\,*_g d \Wd$. 
Combining this identity with the previous one, the thesis follows. 
\end{enumerate}
\end{proof}

In Definition \ref{specialSHF}, the eigenvalue $c$ is assumed to be a real constant. In the next result, we obtain a bound on its possible values that depends only on the condition \ref{iiispec}).  

\begin{proposition}\label{boundc}
Let $(\omega,\psip)$ be a symplectic half-flat $\SU(3)$-structure such that $|d\Wd|^2 = c\, |\Wd|^2$, for some real number $c$.  
Then, $c\geq \frac{|\Wd|^2}{4}$ and the equality holds if and only if $(\omega,\psip)$ has Hermitian Ricci tensor. 
\end{proposition}
\begin{proof}
From the decomposition $d\Wd = \frac{|\Wd|^2}{4} \psip + \gamma$ given in Lemma \ref{dwdlemma}, we deduce that 
\[
|d\Wd|^2 = \frac{|\Wd|^4}{4} + |\gamma|^2, 
\]
as the summands are $g$-orthogonal and $|\psip|=2$. 
Then, the identity $|d\Wd|^2 = c\, |\Wd|^2$ implies $c \geq \frac{|\Wd|^2}{4}$. Moreover, $c  =  \frac{|\Wd|^2}{4}$ if and only if $|\gamma|=0$. 
This last condition is equivalent to $\Ric_g$ being Hermitian by Proposition \ref{JRic}. 
\end{proof}

Thus, among all SHF structures satisfying the condition \ref{iiispec}) of Definition \ref{specialSHF}, SHF structures with Hermitian Ricci tensor 
can be characterized as those for which the eigenvalue $c$ attains the minimum possible value. 

\smallskip

As we observed in \cite{FiRa}, SHF structures satisfying the conditions \ref{ispec})--\ref{iiispec}) of Definition \ref{specialSHF} occur on every unimodular Lie algebra admitting SHF structures. 
Thus, there exist various examples of compact 6-manifolds with a special SHF structure that are locally homogeneous. 
In the next proposition, we summarize some relevant properties of the examples considered in \cite{FiRa}.  
We refer the reader to Appendix \ref{appendix} for further details.  

\begin{proposition}[\cite{FiRa}]\label{SpecialFR}
Every unimodular Lie algebra admitting symplectic half-flat $\SU(3)$-structures also admits special symplectic half-flat $\SU(3)$-structures. 
The values of $c$ and $|\Wd|^2$ for an example of such a structure on each Lie algebra are summarized in the following list: 
\[
\renewcommand\arraystretch{1.5}
\begin{array}{rll}
\fre(1,1)\oplus\fre(1,1):			& c =	 2,			\quad& |\Wd|^2 = 8;\\
\frg_{5,1}\oplus\R:				& c =	 2,			\quad& |\Wd|^2 = 2; \\
A_{5,7}^{-1,-1,1}\oplus\R:			& c = 4,			\quad& |\Wd|^2 = 8;\\
A_{5,17}^{a,-a,1}\oplus\R:			& c =4a^2,		\quad& |\Wd|^2 = 8a^2;\\
\frg_{6,N3}:					& c = 6,			\quad& |\Wd|^2 = 6;\\
\frg_{6,38}^{0}:					& c = 6,			\quad& |\Wd|^2 = 6;\\
\frg_{6,54}^{0,-1}:				& c = 2,			\quad& |\Wd|^2 = 6;\\
\frg_{6,118}^{0,-1,-1}:			& c = 4,			\quad& |\Wd|^2 = 8.
\end{array}
\]
\end{proposition}

\begin{remark}\label{cWrelation}
The example on $\fre(1,1)\oplus\fre(1,1)$ has Hermitian Ricci tensor by Proposition \ref{boundc}, as $c = \tfrac14 |\Wd|^2$. 
In the remaining examples, the possible values of $c$ are $|\Wd|^2$, $\tfrac12 |\Wd|^2$ and $\tfrac13 |\Wd|^2$. 
\end{remark}

Summing up, special SHF structures exist on all known homogeneous and compact locally homogeneous spaces admitting invariant SHF structures, and the known examples have 
$c\in \left\{ |\Wd|^2, \tfrac12 |\Wd|^2, \tfrac13 |\Wd|^2, \tfrac14 |\Wd|^2 \right\}$, with $c = \tfrac14 |\Wd|^2$ if and only if the SHF structure has Hermitian Ricci tensor.  
This naturally leads to the question whether there exist further examples of 6-manifolds that admit a special SHF structure and that are not (locally) homogeneous. 
However, it may also happen that (some of) the conditions \ref{ispec})--\ref{iiispec}) force a SHF 6-manifold $(M,\omega,\psip)$ to be locally homogeneous. 
We leave this as an open problem. 

On the other hand, there exist examples of SHF structures that are not special according to Definition \ref{specialSHF}. 
These examples include the family of SHF structures  on the 6-torus constructed in \cite{DeBTom}, the complete cohomogeneity one SHF structure on $\mathbb{S}^3\times\R^3$ obtained in \cite{PoRa2}, 
and the SHF structure on the Lie algebra $\frg_{6,54}^{0,-1}$ given in \cite[Ex.~3.1]{ToVe}. 

\smallskip

Going back to Type IIA geometry, we introduce the following definition for those Type IIA geometries inducing special SHF structures. 
\begin{definition}\label{speciaIIAdef}
A Type IIA geometry $(\omega,\f)$ is {\em special} if $\f$ has constant norm and the corresponding symplectic half-flat $\SU(3)$-structure is special. 
\end{definition}
Since $\Wd = d^*\left(\frac{2}{|\f|}\f\right)$, the conditions \ref{ispec})--\ref{iiispec}) of Definition \ref{specialSHF} can be easily restated in terms of $d^*\f$.


\section{Two useful lemmas}\label{LemmaSect}
In this section, we discuss some results that are related to the properties of a special SHF structure and that might be of some interest in their own right. 

\smallskip

Let us consider a six-dimensional vector space $V$ with an SU(3)-structure $(\omega,\psip)$ inducing the complex structure $J$ and the metric $g$. 
The space $\Sym(V)$ of symmetric endomorphisms of $V$ decomposes into irreducible SU(3)-modules as follows
\[
\Sym(V) = \R\, \mathrm{Id} \oplus \Sym^{\sst+}_{\sst0}(V) \oplus  \Sym^{\sst-}(V), 
\]
where 
\[
\begin{split}
\Sym^{\sst+}_{\sst0}(V)	&\coloneqq \left\{A \in \Sym(V) \st AJ=JA \mbox{ and } \tr(A)=0 \right\},\\
\Sym^{\sst-}(V)			&\coloneqq \left\{S \in \Sym(V) \st  SJ=-JS \right\}.
\end{split}
\]
Moreover, we have the following isomorphisms of SU(3)-representations
\[
\begin{split}
\Sigma_8:\Sym^{\sst+}_{\sst0}(V) \rightarrow\Lambda^2_8V^*,		&\quad  A \mapsto g(AJ\cdot, \cdot), \\
\Sigma_{12}: \Sym^{\sst-}(V) \rightarrow \Lambda^3_{12}V^*,		&\quad  S \mapsto S_*\psip = -\psip(S\cdot,\cdot,\cdot) -\psip(\cdot,S\cdot,\cdot)-\psip(\cdot,\cdot,S\cdot), 
\end{split}
\]
see for instance \cite{BeVe,MNS}. 

\smallskip

The first result we discuss gives some insights on the condition \ref{iispec}) of Definition \ref{specialSHF} and singles out a basis that might be useful for computations involving 
special SHF structures. 
\begin{lemma}\label{Lemma1}
Consider two endomorphisms $A\in \Sym^{\sst+}_{\sst0}(V)$ and $S\in\Sym^{\sst-}(V)$, and let $\sigma \coloneqq \Sigma_8(A)\in\Lambda^2_8V^*$ and $\rho \coloneqq \Sigma_{12}(S)=S_*\psip \in \Lambda^3_{12}V^*$.  
Then, $A$ and $S$ commute if and only if $\sigma\W\rho=0$. 
Whenever this happens, there exists a unitary basis of $V$ formed by common eigenvectors of $A$ and $S$. 
\end{lemma}
\begin{proof}
Let $\mathcal{B}=(e_1,\ldots,e_6)$ be a $g$-orthonormal basis of $V$ that is adapted to $(\omega,\psip)$, and denote by $\mathcal{B}^*=(e^1,\ldots,e^6)$ its dual basis. 
Two generic symmetric endomorphisms $A\in\Sym^{\sst+}_{\sst0}(V) $ and $S \in \Sym^{\sst-}(V)$ have the following matrix representations with respect to the basis $\mathcal{B}$ 
\[
A = 
\begin{pmatrix}
a_{11} & 0 & a_{13} & a_{14} & a_{15} & a_{16} \\
0 & a_{11} & -a_{14} & a_{13} & -a_{16} & a_{15} \\
a_{13} & -a_{14} & a_{33} & 0 & a_{35} & a_{36} \\
a_{14} & a_{13} & 0 & a_{33} & -a_{36} & a_{35} \\
a_{15} & -a_{16} & a_{35} & -a_{36} & -a_{11}-a_{33} & 0 \\
a_{16} & a_{15} & a_{36} & a_{35} & 0 & -a_{11}-a_{33} 
\end{pmatrix},
\]
\[
S =
\begin{pmatrix}
s_{11} & s_{12} & s_{13} & s_{14} & s_{15} & s_{16} \\
s_{12} & -s_{11} & s_{14} & -s_{13} & s_{16} & -s_{15} \\
s_{13} & s_{14} & s_{33} & s_{34} & s_{35} & s_{36} \\
s_{14} & -s_{13} & s_{34} & -s_{33} & s_{36} & -s_{35} \\
s_{15} & s_{16} & s_{35} & s_{36} & s_{55} & s_{56} \\
s_{16} & -s_{15} & s_{36} & -s_{35} & s_{56} & -s_{55} 
\end{pmatrix}, 
\]
for some real numbers $a_{ij}$ and $s_{kl}$. 

Using the correspondences $\sigma = \Sigma_8(A)$ and $\rho = \Sigma_{12}(S) = S_*\psip$, we obtain
\[
\begin{split}
\sigma 	= &\ a_{1 1} e^{12} - a_{1 4}  e^{13} + a_{1 3} e^{14}- a_{1 6} e^{15}+ a_{1 5} e^{16}- a_{1 3} e^{23} - a_{1 4} e^{24} - a_{1 5} e^{25}- a_{1 6}e^{26} \\ 
		& + a_{3 3}e^{34}- a_{3 6} e^{35}+ a_{3 5} e^{36} - a_{3 5} e^{45}- a_{3 6} e^{46}+ (-a_{1 1} - a_{3 3}) e^{56}, \\
\rho		= &\  2s_{1 6} e^{123} + 2s_{1 5}e^{124} - 2s_{1 4}e^{125}- 2s_{1 3}e^{126} - 2s_{3 6}e^{134}+ (-s_{5 5} - s_{1 1} - s_{3 3})e^{135}\\
		& + (-s_{5 6} + s_{1 2} + s_{3 4})e^{136}+ (s_{5 6} + s_{1 2} - s_{3 4})e^{145} + (-s_{5 5} + s_{1 1} - s_{3 3})e^{146} + 2s_{3 6}e^{156}\\ 
		& - 2s_{3 5}e^{234} + (s_{5 6} - s_{1 2} + s_{3 4})e^{235}+ (-s_{5 5} - s_{1 1} + s_{3 3})e^{236} + (s_{5 5} - s_{1 1} - s_{3 3})e^{245} \\ 
		&+ (s_{5 6} + s_{1 2} + s_{3 4})e^{246}+ 2s_{3 5}e^{256} + 2s_{1 4}e^{345}+ 2s_{1 3}e^{346}- 2s_{1 6}e^{356}- 2s_{1 5}e^{456}.
\end{split}
\]
Now, a computation shows that 
\[
\sigma \W \rho = p_1 e^{12345} + p_2 e^{12346} + p_3 e^{12356} + p_4 e^{12456} + p_5 e^{13456} + p_6 e^{23456}, 
\]
where $p_1,\ldots,p_6$ are certain homogeneous polynomials of degree 2 in $a_{ij}$ and $s_{kl}$,    
and that the matrix representation of the skew symmetric endomorphism $AS-SA$ with respect to the basis $\mathcal{B}$ is given by 
\[
AS-SA = \frac12
\begin{pmatrix}
0 & 0 & p_2 & p_1 & -p_4 & -p_3 \\
0 & 0& p_1 & -p_2 & -p_3 & p_4 \\
-p_2 & -p_1 & 0 & 0 & p_6 & p_5 \\
-p_1 & p_2 & 0 & 0 & p_5 & -p_6 \\
p_4 & p_3 & -p_6 & -p_5 & 0 & 0 \\
p_3 & -p_4 & -p_5 & p_6 & 0 & 0
\end{pmatrix}. 
\]
From this, we see that $A$ and $S$ commute if and only if $\sigma\W\rho=0$. 
Consequently, whenever $\sigma\W\rho=0$, there exists a basis $\tilde{\mathcal{B}} = (\tilde{e}_1,\ldots,\tilde{e}_6)$ of $V$ formed by common eigenvectors of $A$ and $S$ that is unitary, namely 
$\tilde{\mathcal{B}}$ is $g$-orthonormal and $J\tilde{e}_{2k-1} = \tilde{e}_{2k}$, for $k=1,2,3$. 
More in detail, the spectrum of $A$ is of the form $(\lambda_1,\lambda_1,\lambda_2,\lambda_2,\lambda_3,\lambda_3)$, for some $\lambda_1,\lambda_2,\lambda_3\in\R$ with $\lambda_1+\lambda_2+\lambda_3=0$, 
the spectrum of $S$ is of the form $(\mu_1,-\mu_1,\mu_2,-\mu_2,\mu_3,-\mu_3)$, for some $\mu_1,\mu_2,\mu_3\in\R$, and we have
\[
\begin{split}
&A\te_1=\lambda_1\te_1,~ A\te_2=\lambda_1\te_2,~ A\te_3=\lambda_2\te_3,~ A\te_4=\lambda_2\te_4,~ A\te_5=\lambda_3\te_5,~ A\te_6=\lambda_3\te_6,\\
&S\te_1=\mu_1\te_1,~ S\te_2=-\mu_1\te_2,~ S\te_3=\mu_2\te_3,~ S\te_4=-\mu_2\te_4,~ S\te_5=\mu_3\te_5,~ S\te_6=-\mu_3\te_6. 
\end{split}
\]
Since the basis $\tilde{\mathcal{B}}$ is unitary, the expressions of the forms $\omega$, $\psip$ and $\psim$ with respect to the dual basis $\tilde{\mathcal{B}}^* = (\tilde{e}^1,\ldots,\tilde{e}^6)$ are the following
\begin{equation}\label{u3frame}
\begin{split}
\omega	&= \te^{12}+\te^{34}+\te^{56}, \\
\psip		&= \cos(\theta)\left(\te^{135}-\te^{146}-\te^{236}-\te^{245}\right) + \sin(\theta) \left(\te^{136}+\te^{145}+\te^{235}-\te^{246}\right),\\
\psim	&= \cos(\theta) \left(\te^{136}+\te^{145}+\te^{235}-\te^{246}\right) -\sin(\theta)\left(\te^{135}-\te^{146}-\te^{236}-\te^{245}\right),
\end{split}
\end{equation}
for some $\theta\in[0,2\pi)$. In particular, $\theta=0$ if and only if the basis $\tilde{\mathcal{B}}$ is special unitary, namely if and only if it is adapted to the SU(3)-structure $(\omega,\psip)$. 
\end{proof} 

\smallskip

From the previous discussion, we know that, at each point $x$ of a manifold $M$ with a special SHF structure $(\omega,\psip)$,  
the component of $d\Wd$ in $\Lambda^3_{12}(T^*_xM)$ is of the form $S_*\psip$, for some $S\in\Sym^{\sst-}(T_xM)$. 
In the known examples where such a component is not zero, i.e., those whose Ricci tensor is not Hermitian, the endomorphism $S$ satisfies some constraints.  
In detail, one of the following occurs:
\begin{enumerate}[$\bullet$]
\item $S$ has rank 2 and its spectrum is of the form $(0,0,0,0,\mu,-\mu)$, with $\mu= \pm\tfrac14\,|\Wd|^2$; 
\item $S$ has rank 6 and its spectrum is of the form  $(\mu,-\mu,\mu,-\mu,\mu,-\mu)$, with either $\mu = - \tfrac14\,|\Wd|^2$ or $\mu=\tfrac{1}{12}\,|\Wd|^2$,
\end{enumerate}
see Appendix \ref{appendix} for the explicit expressions. 
The first possibility holds for the examples on the Lie algebras $A_{5,7}^{-1,-1,1}\oplus\R$, $A_{5,17}^{a,-a,1}\oplus\R$, $\frg_{6,118}^{0,-1,-1}$, 
while the second possibility holds for the examples on $\frg_{5,1}\oplus\R$, $\frg_{6,N3}$, $\frg_{6,38}^{0}$ ($\mu = - \tfrac14\,|\Wd|^2$) and $\frg_{6,54}^{0,-1}$ ($\mu=\tfrac{1}{12}\,|\Wd|^2$). 
Notice that, in both cases, there exists a basis formed by eigenvectors of $S$ that is adapted to the SU(3)-structure $(\omega,\psip)$. 

\smallskip

Let us consider again a vector space $V$ with an SU(3)-structure $(\omega,\psip)$.  
As we observed in the proof of Lemma \ref{Lemma1}, the spectrum of an endomorphism $S\in\Sym^{\sst-}(V)$ is of the form $(\mu_1,-\mu_1,\mu_2,-\mu_2,\mu_3,-\mu_3)$, 
for some real numbers $\mu_1,\mu_2,\mu_3$. 
We now show a result that, when applied to a manifold with a special SHF structure, gives a (pointwise) relation among the constant $c$ appearing in Definition \ref{specialSHF}, the norm of the torsion form $\Wd$, 
and the eigenvalues of $S$. More generally, this result holds for every SHF structure satisfying condition \ref{iiispec}) of Definition \ref{specialSHF}. 
\begin{lemma}\label{eigenconstr}
Let $\rho = S_*\psip\in\Lambda^3_{12}V^*$, where $S\in\Sym^{\sst-}(V)$, and let $w$ and $c$ be non-zero real constants. 
Consider the $3$-form $\beta = \frac{w^2}{4}\psip + \rho \in \Lambda^3_{\mathrm{Re}}V^*\oplus\Lambda^3_{12}V^*$, and assume that $|\beta|^2 = c\,w^2$. Then
\[
\mu_1^2 + \mu_2^2+\mu_3^2 = \frac14\, w^2 \left(c-\frac14 w^2 \right), 
\]
where $(\mu_1,-\mu_1,\mu_2,-\mu_2,\mu_3,-\mu_3)$ is the spectrum of $S$. 
\end{lemma}
\begin{proof}
Since $S$ is symmetric and anticommutes with $J,$ there exists a unitary basis of $V$ formed by eigenvectors of $S$. 
Let $\mathcal{B}=(e_1,\ldots,e_6)$ be such a basis and denote by $\mathcal{B}^*=(e^1,\ldots,e^6)$ its dual basis. 
Recall that $\mathcal{B}$ is $g$-orthonormal and that $J{e}_{2k-1} = {e}_{2k}$, for $k=1,2,3$.
Moreover, we have
\[
Se_1=\mu_1e_1,~ Se_2=-\mu_1e_2,~ Se_3=\mu_2e_3,~ Se_4=-\mu_2e_4,~ Se_5=\mu_3e_5,~ Se_6=-\mu_3e_6. 
\]
As we recalled in the proof of Lemma \ref{Lemma1}, there exists some $\theta\in[0,2\pi)$ such that $\psip$ and $\psim$ can be written as in \eqref{u3frame} with respect to $\mathcal{B}^*$. 

Since $*_g\beta = \frac{w^2}{4}\psim + *_g\rho$ and the spaces $\Lambda^3_{\mathrm{Re}}V^*$ and $\Lambda^3_{12}V^*$ are $g$-orthogonal, the condition $|\beta|^2 = c\,w^2$ can be rewritten as follows
\[
c\,w^2 = |\beta|^2 = *_g(\beta\W *_g\beta) = \frac14\,w^4 + |\rho|^2. 
\]
To get the thesis, it is then sufficient to compute $|\rho|^2$. With respect to the basis $\mathcal{B}^*$, we have 
\[\begin{split}
\rho 	&=   -\cos(\theta) \left[(\mu_1+\mu_2+\mu_3)\,e^{135} + (-\mu_1+\mu_2+\mu_3)\,e^{146} + (\mu_1-\mu_2+\mu_3)\,e^{236}\right.\\
	&\quad \left. +(\mu_1+\mu_2-\mu_3)\,e^{245}\right] -\sin(\theta) \left[ (\mu_1+\mu_2-\mu_3)\, e^{136} + (\mu_1-\mu_2+\mu_3)\, e^{145}  \right.\\
	&\quad \left. + (-\mu_1+\mu_2+\mu_3)\, e^{235}+ (\mu_1+\mu_2+\mu_3)\,e^{246}\right], 
\end{split}
\]
and from this we obtain $|\rho|^2 = 4\left(\mu_1^2 + \mu_2^2 + \mu_3^2 \right)$. 
\end{proof}

From the previous lemma, we see that if $S$ has spectrum $(0,0,0,0,\mu,-\mu)$, for some $\mu\in\R$,  then 
\begin{equation}\label{murank2}
\mu^2 = \frac14\, w^2 \left(c-\frac14 w^2 \right), 
\end{equation}
while if $S$ has spectrum $(\mu,-\mu,\mu,-\mu,\mu,-\mu)$, then 
\begin{equation}\label{murank6}
\mu^2 = \frac{1}{12}\, w^2 \left(c-\frac14 w^2 \right). 
\end{equation}
Moreover, $S=0\in\Sym^{\sst-}(V)$ if and only if $c=\frac14 w^2$. This is consistent with Proposition \ref{boundc}.


\section{The source-free type IIA flow starting at a special Type IIA geometry}\label{FlowSect}

We now focus on the source-free Type IIA flow starting at a special Type IIA geometry $(\omega,\fz)$ on a 6-manifold $M$
\begin{equation}\label{IIAflow}
\begin{cases}
\frac{\partial}{\partial t} \f(t) =  dJ_td^{*_t}\left(|\f(t)|_{g(t)}^2\f(t) \right),  \\
\f(0) = \f_{\sst0}, 
\end{cases}
\end{equation}
where we let $J_t\coloneqq J_{\f(t)}$. Recall that, by definition, the norm of $\fz$ is assumed to be constant. 
In the following, we shall use the subscript 0 to denote all tensors induced by $(\omega,\fz)$ but the torsion form of the corresponding SHF structure $(\omega,\psip_{\sst0})$  
\begin{equation}\label{w2special}
\Wd = \frac{2}{|\fz|_{\sst0}}d^{*_0}\fz = d^{*_0} \left( \frac{2}{|\fz|_{\sst0}}\fz\right). 
\end{equation} 
Here and henceforth, the symbol $|\cdot|_{\sst0}$ denotes a norm induced by $g_{\sst0}$. 

\smallskip

We begin considering the case where the metric $g_{\sst0}$ has Hermitian Ricci tensor. 
Under this assumption, we can show that the solution to the Type IIA flow \eqref{IIAflow} evolves only by scaling the initial datum $\fz$, namely it is {\em self-similar}. 
\begin{theorem}\label{JHermFlow}
Let $(\omega,\fz)$ be a Type IIA geometry such that $\Fz \coloneqq |\fz|_{\sst0}$ is constant and the Ricci tensor of $g_{\sst0}$ is Hermitian. Then, a solution to the source-free Type IIA flow starting at $\fz$ is
\[
\f(t) = \frac{1}{\sqrt{1 - 2 c \Fz^2 t}} \, \fz, 
\]
where $c=\tfrac14|\Wd|^2_{\sst0}$, and it exists for all $t \in \left(-\infty, \frac{1}{2 c \Fz^2}\right)$. This is the unique solution to the flow when the manifold is compact. 
Moreover, $\frac{1}{|\f(t)|_{g(t)}}\f(t) = \frac{1}{\Fz}\fz$, and the $\SU(3)$-structure corresponding to $(\omega,\f(t))$ is constant along the flow. 
\end{theorem}

\begin{proof}
Since $|\fz|_{\sst0}$ is constant and the Ricci tensor of $g_{\sst0}$ is Hermitian, by Proposition \ref{JRic} we know that $|\Wd|^2_{\sst0}$ is constant and we obtain 
\[
\Delta_{\gz}\fz = \frac{|\Wd|^2_{\sst0}}{4} \fz = c\,\fz. 
\]
This suggests the following Ansatz for the solution of \eqref{IIAflow}
\[
\f(t) = h(t)\,\fz, 
\]
where $h(t)$ is a real valued smooth function defined in some neighborhood of $0\in\R$ and such that $h(0)=1$. 
Since $\f(t)$ is proportional to $\fz$, the almost complex structure $J_t$ induced by $(\omega,\f(t))$ coincides with $\Jz$. Consequently, $g(t) = \gz$ as long as $\f(t)$ exists, and we see that 
$|\f(t)|_{g(t)}^2 = \Fz^2\,h(t)^2$ is a function of $t$ only. 
Now, since $\fz$ is closed and $d^{*_0}\fz$ is $\Jz$-invariant, the RHS of the flow equation in \eqref{IIAflow} can be rewritten as follows
\[
dJ_td^{*_t}\left(|\f(t)|_{g(t)}^2\f(t) \right) = \Fz^2\,h(t)^3\, d\Jz d^{*_0} \fz = \Fz^2\,h(t)^3\, \Delta_{\gz}\fz =  \Fz^2\,h(t)^3\,c\,\fz.
\]
Consequently, under our Ansatz, the Type IIA flow starting at $\fz$ is equivalent to the following initial value problem for $h(t)$ 
\[
\begin{cases}
\frac{d}{dt}h(t) =  \Fz^2\,c\,h(t)^3, \\
h(0)=1.
\end{cases}
\]
From this, the thesis follows. 
\end{proof}

Examples where the previous theorem applies include the homogeneous spaces $\SU(2,1)/\mathrm{T}^2$ and  $\SO_{\sst0}(4,1)/\U(2)$
with their invariant SHF structures described in \cite{PoRa1}, and the Lie algebra $\fre(1,1)\oplus \fre(1,1)$ with the SHF structure given in \cite[Thm.~3.5]{ToVe}. 
The solution to the Type IIA flow on this last space was investigated in \cite[Sect.~9.3.2]{FPPZ}, where the authors considered a family of SHF structures including the one with Hermitian Ricci tensor. 

\smallskip

We now focus on the more general case where $(\omega,\fz)$ is a special Type IIA geometry. 
 Let $(\omega,\psip_{\sst0})$ be the special SHF structure corresponding to $(\omega,\fz)$. 
Then, $d\psim_{\sst0} = \Wd\W\omega = -*_{g_0}\Wd$, where the torsion form $\Wd$ is given by \eqref{w2special} 
and it satisfies the conditions $\Delta_{g_0}\Wd = c\,\Wd$,  $d\Wd \W \Wd=0$, and $|d\Wd|^2_{\sst0} = c\, |\Wd|_{\sst0}^2$. 
According to the discussion in Section \ref{LemmaSect}, 
the component of $d\Wd$ in $\Omega^3_{12}(M)$ is of the form $S_*\psip$, for some section $S$ of the bundle $\Sym^{\sst-}(TM)$ of $g_{\sst0}$-symmetric endomorphisms anticommuting with $\Jz.$ 

Assuming that $(M,\omega,\fz)$ is locally homogeneous, we get the additional condition $|\Wd|_{\sst0}\in\R_{\sst+}$  
and we see that $S_x\in\Sym^{\sst-}(T_xM)$ has the same spectrum at every point $x$ of $M.$  
As we observed in Section \ref{LemmaSect}, all known examples of this type have some restrictions on $c$ and on $S$. 
In detail, if $S$ is not zero, then its rank is either 2 or 6 and it has two non-zero eigenvalues $\pm\mu$, for some $\mu\in\R\smallsetminus\{0\}$. 
Moreover, $c$, $\mu$ and $|\Wd|_{\sst0}$ are related by one of the equations \eqref{murank2} and \eqref{murank6}. 
The possible values of $c$ are $|\Wd|^2_{\sst0}$, $\tfrac12 |\Wd|^2_{\sst0}$ and $\tfrac13 |\Wd|^2_{\sst0}$, and the corresponding values of $\mu$ are $-\tfrac14|\Wd|^2_{\sst0}$, $\pm\tfrac14|\Wd|^2_{\sst0}$,   
and  $\tfrac{1}{12}|\Wd|^2_{\sst0}$, respectively. 
On the other hand, $S$ is zero if and only if the SHF structure $(\omega,\psip_{\sst0})$ has Hermitian Ricci tensor, and thus $c=\tfrac14|\Wd|^2_{\sst0}$. 

Motivated by Theorem \ref{JHermFlow}, we look for a suitable Ansatz that allows us to determine the solution of the flow \eqref{IIAflow} when $(M,\omega,\fz)$ is locally homogeneous and 
the special SHF structure satisfies one of the previous conditions. 
We consider the following 
\begin{equation}\label{solspec}
\f(t) = \fz + \frac{a(t)}{\Fz}\, \Delta_{g_0}\fz, 
\end{equation}
where $\Fz \coloneqq |\fz|_{\sst0}$ and $a(t)$ is a real valued smooth function defined in some connected interval $I\subseteq \R$ containing $0$ and such that $a(0)=0$. 
Notice that the solution determined in Theorem \ref{JHermFlow} can be written in this form by choosing 
\begin{equation}\label{aJRic}
a(t)  = \frac{\Fz}{c}\left(h(t) -1\right) =  \frac{\Fz}{c}\left( \frac{1}{\sqrt{1-2c\Fz^2 t}} - 1\right). 
\end{equation}

Since $\fz$ is closed and its norm $\Fz$ is constant, using \eqref{w2special} we can rewrite \eqref{solspec} as follows
\[
\f(t) = \fz + \frac{a(t)}{\Fz}\, \Delta_{g_0}\fz = \fz + \frac{a(t)}{2} \, d\Wd. 
\]
Thus, $\f(t)$ is closed and primitive with respect to $\omega$ (cfr.~Lemma \ref{dwdlemma}).  
Moreover, $\f(t)$ is positive when $t$ is sufficiently close to $0$, as $\fz$ is positive, $a(0)=0$, and being positive is an open condition.
Therefore, the pair $(\omega,\f(t))$ defines a Type IIA geometry such that $\f(0)=\fz$, for $t$ sufficiently close to $0$. 
In particular, we can consider the corresponding SU(3)-structure $(\omega,\psip_t)$. 

Let $F_t\coloneqq |\f(t)|_{g(t)}$. By definition, we have
\[
\psip_t = \frac{2}{F_t}\,\f(t) =  \frac{\Fz}{F_t}\,\psip_{\sst0} + \frac{a(t)}{F_t}\, d\Wd =  \frac{\Fz}{F_t}\left(\psip_{\sst0} + \frac{a(t)}{\Fz}\, d\Wd\right).
\]
Since $(M,\omega,\fz)$ is locally homogeneous, $F_t$ is a function of $t$ only and, consequently, $\psip_t$ is closed. Thus, the SU(3)-structure $(\omega,\psip_t)$ is SHF.  

The restrictions on $c$ and $S$ mentioned above allow us to determine the expression of $F_t$ and $\psim_t = J_t\psip_t$ with the help of the results discussed in Section \ref{LemmaSect}.  
In detail, we have the following. 
\begin{proposition}\label{lemmapsimt}
Assume that at some point $x$ of $M$ (and thus at each point of the manifold) the endomorphism $S\in\Sym^{\sst-}(T^*_xM)$ satisfies one of the following conditions:
\begin{enumerate}[a)]
\item\label{caseA} $S$ has rank $2$ and spectrum $(0,0,0,0,\mu,-\mu)$, with either $\mu = \frac14 |\Wd|_{\sst0}^2$ or $\mu = - \frac14 |\Wd|_{\sst0}^2$; 
\item\label{caseB} $S$ has rank $6$ and spectrum $(\mu,-\mu,\mu,-\mu,\mu,-\mu)$, with either $\mu=-\frac14 |\Wd|_{\sst0}^2$ or $\mu=\frac{1}{12} |\Wd|_{\sst0}^2$, 
and there exists a basis of $T_xM$ formed by eigenvectors of $S$ that is adapted to the special SHF structure $(\omega,\psip_{\sst0})$. 
\end{enumerate} 
Then, 
\begin{equation}\label{Fteq}
F_t = \Fz \left(1 + \frac{c}{\Fz}\,a(t) \right)^{\mbox{$\frac{|\Wd|^2_{\sst0}}{4c}$}}, 
\end{equation}
and
\begin{equation}\label{psimteq}
\psim_t = J_t\psip_t =  \frac{F_t}{\Fz}\left(\psim_{\sst0} - \frac{a(t)}{\Fz + c\,a(t)} *_{g_0} d\Wd \right), 
\end{equation}
where the constants $c$, $\mu$ and $|\Wd|_{\sst0}^2$ are related by equation \eqref{murank2} in the case \ref{caseA}) and by equation \eqref{murank6} in the case \ref{caseB}). 
\end{proposition}
\begin{proof}
The identities can be shown via a pointwise computation with respect to a suitable basis of the tangent space. 

In the case \ref{caseA}), we have $c=\frac12|\Wd|^2_{\sst0} = \pm2\mu$, 
and we consider a basis $\mathcal{B}=(e_1,\ldots,e_6)$ of $T_xM$ formed by eigenvectors of $S$ that is unitary with respect to the almost Hermitian structure $(g_{\sst0},\Jz)$.  
It is not restrictive to assume that $Se_1=\mu e_1$, $Se_2 =-\mu e_2$ and $Se_k=0$ otherwise. 
With respect to the dual basis $\mathcal{B}^*=(e^1,\ldots,e^6)$ of $\mathcal{B}$, the 3-form $\psip_{\sst0}$ can be written as in \eqref{u3frame} for some $\theta\in[0,2\pi)$. 
We can then determine the expression of $d\Wd = \frac14  |\Wd|_{\sst0}^2\psip_{\sst0} + S_*\psip_{\sst0}$ with respect to $\mathcal{B}^*$ in the same fashion as in the proof of Lemma \ref{eigenconstr}.  

Since $\f(t) = \fz + \frac{a(t)}{2} \, d\Wd = \frac{\Fz}{2}\,\psip_{\sst0} + \frac{a(t)}{2} \, d\Wd$, at $x$ we have
\[
\begin{array}{rcl}
\f(t) 	&=&  \frac{\cos(\theta)}{2}\left(  \left(\Fz + a(t)\left(\frac{|\Wd|^2_{\sst0}}{4}-\mu \right)  \right) \left( e^{135}-e^{146}\right) -  \left(\Fz + a(t)\left(\frac{|\Wd|^2_{\sst0}}{4}+\mu \right)  \right) \left( e^{236}+e^{245}\right) \right) 
\vspace{0.1cm} \\ 
& &  +\frac{\sin(\theta)}{2}\left(  \left(\Fz + a(t)\left(\frac{|\Wd|^2_{\sst0}}{4}-\mu \right)  \right)  \left( e^{136}+e^{145}\right)+ \left(\Fz + a(t)\left(\frac{|\Wd|^2_{\sst0}}{4}+\frac{c}{2} \right)  \right)   \left( e^{235} - e^{246}\right)  \right). 
\end{array}
\] 
Using the general identity from \cite[Lemma 4]{FPPZ}
\begin{equation}\label{lemma4FPPZ}
\sqrt{-P(\f)} = \frac12 |\f|^2, 
\end{equation}
and the definition of the polynomial $P(\f)$, we see that
\[
F_t = \sqrt{  \left(\Fz + a(t)\left(\frac{|\Wd|^2_{\sst0}}{4}-\mu \right)  \right)   \left(\Fz + a(t)\left(\frac{|\Wd|^2_{\sst0}}{4}+\mu \right)  \right) }\,, 
\]
and thus \eqref{Fteq} is verified when $c=\frac12|\Wd|^2_{\sst0} = \pm2\mu$.   
We then compute the almost complex structure $J_t$ induced by $(\omega,\psip_t)$ using the definition recalled in Section \ref{preliminaries}, obtaining  
\[
J_t e_1 = \frac{\Fz + a(t)\left(\frac{|\Wd|^2_{\sst0}}{4}-\mu \right) }{\Fz + a(t)\left(\frac{|\Wd|^2_{\sst0}}{4}+\mu \right) } e_2, \quad 
J_te_2 = - \frac{\Fz + a(t)\left(\frac{|\Wd|^2_{\sst0}}{4}+\mu \right) }{\Fz + a(t)\left(\frac{|\Wd|^2_{\sst0}}{4}-\mu \right)} e_1,  
\] 
and $J_t e_{2k-1} = e_{2k}$, for $k=2,3$. 
It is now possible to compute the expression of $\psim_t = J_t\psip_t$ with respect to $\mathcal{B}^*$ and conclude that \eqref{psimteq} holds when $c=\frac12|\Wd|^2_{\sst0} = \pm2\mu$. 

In the case \ref{caseB}), we have either $c = |\Wd|^2_{\sst0} = -4\mu$ or $c = \frac13|\Wd|^2_{\sst0} = 4\mu$. 
Moreover, by hypothesis there exists a basis $\mathcal{B}=(e_1,\ldots,e_6)$ of $T_xM$ formed by eigenvectors of $S$ that is adapted to $(\omega,\psip_{\sst0})$. 
Thus, we have 
\[
Se_{2k-1}=\mu e_{2k-1},\quad Se_{2k} = -\mu e_{2k}, \quad k=1,2,3, 
\]
and we can write $\psip$ and $\psim$ as in \eqref{adaptedfr} with respect to the dual basis of $\mathcal{B}$. We then compute
\[
F_t = \sqrt[4]{  \left(\Fz + a(t)\left(\frac{|\Wd|^2_{\sst0}}{4}+\mu \right)  \right)^3   \left(\Fz + a(t)\left(\frac{|\Wd|^2_{\sst0}}{4}-3\mu \right)  \right) }\,,
\]
from which we see that \eqref{Fteq} holds. 
We can then proceed in a similar way as in the case \ref{caseA}) computing the expression of $J_t$ and  then checking \eqref{psimteq}, so we omit the detail. 
\end{proof}

Notice that the identities \eqref{Fteq} and \eqref{psimteq} also hold for the solution to the Type IIA flow considered in Theorem \ref{JHermFlow}, 
as in that case $a(t)$ is given by \eqref{aJRic} and $c=\frac14|\Wd|^2_{\sst0}$. 
Thus, the same conclusions of Proposition \ref{lemmapsimt} hold when $S=0$. 

\begin{remark}
The equations \eqref{murank2} and \eqref{murank6} are a consequence of Lemma \ref{eigenconstr}, and thus of the identity $|d\Wd|^2_{\sst0} = c\, |\Wd|_{\sst0}^2$ of Definition \ref{specialSHF}.
When $\psim_t$ is given by \eqref{psimteq}, the normalization condition for the SU(3)-structure $(\omega,\psip_t)$ follows from this identity.    
Indeed, using it together with the general identities \eqref{iddw2}, we have
\[
\begin{split}
\psip_t \W \psim_t 	&= \psip_{\sst0} \W \psim_{\sst0}  - \frac{a(t)}{\Fz + c\,a(t)}\, \psip_{\sst0} \W *_{g_0} d\Wd  + \frac{a(t)}{\Fz}\, d\Wd \W \psim_{\sst0}  - \frac{(a(t))^2}{\Fz(\Fz + c\,a(t))} |d\Wd|_{\sst0}^2\vol_g \\
				&= \psip_{\sst0} \W \psim_{\sst0} +\left( - \frac{a(t)}{\Fz + c\,a(t)}  + \frac{a(t)}{\Fz} - \frac{(a(t))^2}{\Fz(\Fz + c\,a(t))} \,c \right) |\Wd|_{\sst0}^2 \vol_g =  \frac23\, \omega^3. 
\end{split}
\]
 \end{remark}

\smallskip

In the next theorem, we show that the Type IIA flow is equivalent to an initial value problem for the function $a(t)$ under the Ansatz \eqref{solspec} whenever $\psim_t$ is given by \eqref{psimteq}. 
This will allow us to solve the flow starting at any of the known examples of (locally) homogeneous special Type IIA geometries. 
As we will see in the proof, the expression of $F_t$ given in Proposition \ref{lemmapsimt} also follows from the flow equation.

\begin{theorem}\label{SolThm}
Let $(M,\omega,\fz)$ be a locally homogeneous $6$-manifold with an invariant special Type IIA geometry. Consider the $3$-form 
\[
\f(t) = \fz + \frac{a(t)}{\Fz}\, \Delta_{g_0}\fz, 
\]
where $a(t)$ is a real valued function defined in some neighborhood of $0\in\R$ and such that $a(0)=0$, 
and let $(\omega,\psip_t)$ be the SHF structure corresponding to the Type IIA geometry $(\omega,\f(t))$. 
Assume that 
\[
\psim_t = \frac{F_t}{\Fz}\left(\psim_{\sst0} - \frac{a(t)}{\Fz + c\,a(t)} *_{g_0} d\Wd \right).
\]
Then, $\f(t)$ solves the source-free Type IIA flow starting at $\fz$ 
if and only if $a(t)$ solves the initial value problem 
\begin{equation}\label{sysspec}
\left\{
\begin{split}
\frac{d}{dt} a(t) &= \Fz^3\left(1+\frac{c}{\Fz} a(t) \right)^{\mbox{ $\frac{|\Wd|_{\sst0}^2}{c}-1$ }}, \\
a(0) &= 0.
\end{split} \right.
\end{equation}
Moreover, this is the unique solution to the flow starting at $\fz$ if $M$ is compact. 
\end{theorem}

\begin{proof}
We begin computing the expression of the torsion form $\Wd(t)$ of the SHF structure $(\omega,\psip_t)$. 
Since $\Delta_{g_0}\Wd = c\,\Wd$, we have $ d *_{g_0} d \Wd = -c\,*_{g_0}\Wd = c\,\Wd\W\omega$. Therefore
\[
d\psim_t	= \frac{F_t}{\Fz}\left(d\psim_{\sst0} - \frac{a(t)}{\Fz + c\,a(t)} d*_{g_0} d\Wd \right) = \frac{F_t}{\Fz + c\,a(t)} \, \Wd\W\omega. 
\]
The assumption $d\Wd\W\Wd =0$ ensures that $\Wd\in\Omega^2_8(M)$, where this space is now determined by the SU(3)-structure $(\omega,\psip_t)$. 
Indeed, $\Wd\W\omega^2=0$ and $\Wd\W\psip_t = 0$, whence it follows that $J_t\Wd =\Wd$.  
By the uniqueness of the torsion form $\Wd(t)$, we conclude that 
\[
\Wd(t) = \frac{F_t}{\Fz + c\,a(t)} \, \Wd. 
\]

Now, since $F_t$ depends only on $t$, and $d^{*_t}\f(t) = \frac{F_t}{2}\,\Wd(t)$, we see that the RHS of the flow equation in \eqref{IIAflow} becomes 
\[
dJ_td^{*_t}\left(F_t^2\f(t) \right) = \frac{F_t^3}{2}\,d\Wd(t) = \frac12 \frac{F_t^4}{\Fz + c\,a(t)} \, d\Wd,
\]
as $\Wd(t)$ is $J_t$-invariant. On the other hand, we have
\[
\ddt \f(t) = \frac12\,\frac{d}{dt}a(t)\,d\Wd. 
\]
Therefore, the 3-form $\f(t)$ given by \eqref{solspec} solves the  source-free Type IIA flow equation if and only if $a(t)$ solves the ODE
\begin{equation}\label{atfirst}
\frac{d}{dt}a(t) =  \frac{F_t^4}{\Fz + c\,a(t)}.
\end{equation}

We now show that the expression of $F_t$ given in Proposition \ref{lemmapsimt} can also be obtained from the flow equation. 
We know from \cite[(7.28)]{FPPZ} that the function $u(t) = \log F_t^2$ evolves as follows along the  source-free Type IIA flow
\[
\ddt u(t) = e^{u(t)} \left(\Delta_{g(t)} u(t) +2\,|d u(t)|_{g(t)}^2 + |N_{J_t}|^2_{g(t)}\right).
\]
Under our assumptions, $u(t)$ is a function of $t$ only, and this equation becomes 
\begin{equation}\label{flowut}
\ddt u(t) = e^{u(t)}  |N_{J_t}|^2_{g(t)}. 
\end{equation}
From \eqref{NJSHF}, we know that $|N_{J_t}|^2_{g(t)}  =\frac12\, |\Wd(t)|_{g(t)}^2$. 
Since $\vol_{g(t)} = \frac16\, \omega^3 = \vol_{g_0}$ and $\Wd(t) \in\Omega^2_8(M)$, we can compute the norm of $\Wd(t)$ as follows 
\[
\begin{split}
|\Wd(t)|_{g(t)}^2 \vol_{g(t)}	&= \Wd(t) \W *_{g(t)} \Wd(t) = - \Wd(t) \W \Wd(t) \W \omega \\
					&= -\left( \frac{F_t}{\Fz + c\,a(t)}\right)^2 \, \Wd \W \Wd\W\omega = \left( \frac{F_t}{\Fz + c\,a(t)}\right)^2 |\Wd|^2_{\sst0} \vol_{g_0}. 
\end{split}
\] 
Substituting this last expression into the equation \eqref{flowut}, we then obtain
\begin{equation}\label{Ftflow}
\frac{d}{dt} F_t = \frac14\,\frac{F_t^5}{(\Fz+c\, a(t))^2}\,|\Wd|^2_{\sst0}. 
\end{equation}
Combining the ODEs \eqref{atfirst} and \eqref{Ftflow} and using that $a(0)=0$, we get   
\[
F_t = \Fz \left(1 + \frac{c}{\Fz}\,a(t) \right)^{\mbox{$\frac{|\Wd|^2_{\sst0}}{4c}$}}, 
\]
that is precisely \eqref{Fteq}. Consequently, we have 
\[
\frac{d}{dt} a(t) = \frac{F_t^4}{\Fz + c\,a(t)} = \Fz^3\left(1+\frac{c}{\Fz} a(t) \right)^{\mbox{ $\frac{|\Wd|_{\sst0}^2}{c}-1$ }}. 
\]
Finally, the uniqueness of the solution when $M$ is compact follows from \cite[Thm.\td2]{FPPZ}. 
\end{proof}

\smallskip

\begin{remark}
In the statement of Theorem \ref{SolThm} it is possible to replace the hypothesis on $(M,\omega,\fz)$ with one of the following: 
\begin{enumerate}[1)]
\item $(M,\omega,\fz)$ is compact locally homogeneous and the SHF structure $(\omega,\psip_{\sst0})$ satisfies the conditions \ref{ispec}) and \ref{iispec}) of Definition \ref{specialSHF}; 
\item $M=\G$ is a unimodular Lie group with a left-invariant Type IIA geometry $(\omega,\fz)$ 
and the SHF structure $(\omega,\psip_{\sst0})$ satisfies the conditions \ref{ispec}) and \ref{iispec}) of Definition \ref{specialSHF}.
\end{enumerate}
Indeed, Proposition \ref{specialrelations} ensures that the identity $|d\Wd|^2_{\sst0} = c\, |\Wd|_{\sst0}^2$ holds, and the thesis follows from the same argument used in the proof of Theorem \ref{SolThm}. 
\end{remark}

From the proof of Theorem \ref{SolThm}, we see that the norm of the Nijenhuis tensor of the almost complex structure $J_t$ is 
\[
|N_{J_t}|_{g(t)}^2 =\frac12 |\Wd(t)|^2_{g(t)} = \frac{|\Wd|_{\sst0}^2}{2} \left(1+\frac{c}{\Fz}\, a(t) \right)^{\mbox{$\frac{|\Wd|^2_{\sst0}}{2c}$}-2}. 
\]

\smallskip

It is clear that the expression of the solution to the initial value problem \eqref{sysspec} depends on the relation between the constants $c$ and $|\Wd|_{\sst0}^2$. 
In fact, there are three relevant cases giving rise to different types of solutions. 
We describe each possible case separately in the following corollaries, whose proofs follow from direct computations using the results obtained so far. 
As we will see, all possibilities occur in the known examples. 

\begin{corollary}\label{CorEternal}
If $c = \frac12 |\Wd|_{\sst0}^2$, the solution to \eqref{sysspec} is given by 
\[
a(t) = \frac{\Fz}{c} \left( e^{{c\Fz^2t}} - 1 \right).  
\] 
Consequently, the norm of the solution \eqref{solspec} to the Type IIA flow is
\[
|\f(t)|_{g(t)}  = \Fz\, e^{\frac12 c\Fz^2t}, 	 
\] 
and $\f(t)$ exists for all real times, i.e., it is eternal. Moreover, the norm of the Nijenhuis tensor of $J_t$ is
 \[
|N_{J_t}|_{g(t)}^2 = \frac{|\Wd|_{\sst0}^2}{2}\, e^{-c\Fz^2t}. 
\]
\end{corollary}

\begin{corollary}\label{CorAncient}
If $\frac14 |\Wd|_{\sst0}^2  \leq c < \frac12 |\Wd|_{\sst0}^2$, the solution to \eqref{sysspec} is given by 
\[
a(t) = \frac{\Fz}{c}\left[\left(2c-|\Wd|_{\sst0}^2 \right)\Fz^2 t +1 \right]^{ \mbox{ $\frac{ c }{2c-|\Wd|_{\sst0}^2}$ } } -  \frac{\Fz}{c}.
\] 
Consequently, the norm of the solution \eqref{solspec} to the Type IIA flow is
\[
|\f(t)|_{g(t)}  = \Fz \left[\left(2c-|\Wd|_{\sst0}^2 \right)\Fz^2 t +1 \right]^{ \mbox{ $\frac{ |\Wd|_{\sst0}^2 }{4\left(2c-|\Wd|_{\sst0}^2\right)}$ } },
\]
and $\f(t)$ exists for $t\in \left( -\infty, \frac{1}{\left(|\Wd|_{\sst0}^2 -2c\right)\Fz^2 }\right)$, i.e., it is ancient. Moreover, the norm of the Nijenhuis tensor of $J_t$ is
\[
|N_{J_t}|_{g(t)}^2 = \frac{|\Wd|_{\sst0}^2}{2} \left[\left(2c-|\Wd|_{\sst0}^2 \right)\Fz^2 t +1 \right]^{ \mbox{ $\frac{ |\Wd|_{\sst0}^2 -4c }{2\left(2c-|\Wd|_{\sst0}^2\right)}$ } }.
\]
\end{corollary}

\begin{corollary}\label{CorImmortal}
If $c > \frac12 |\Wd|_{\sst0}^2$, the expressions of $a(t)$, $|\f(t)|_{g(t)}$ and $|N_{J_t}|_{g(t)}^2$ are those given in Corollary \ref{CorAncient}. 
In this case, the solution $\f(t)$ exists for $t\in \left( \frac{1}{\left(|\Wd|_{\sst0}^2 -2c\right)\Fz^2 } , +\infty\right)$, i.e., it is immortal. 
\end{corollary}

In the previous corollaries, the maximal time interval where the solution $\f(t)$ exists is the maximal connected real interval containing $0$ where $|\f(t)|_{g(t)}$ is different from zero. 
Indeed, the general identity from \cite[Lemma 4]{FPPZ} recalled in \eqref{lemma4FPPZ} ensures that the primitive 3-form $\f(t)$ is positive as long as $|\f(t)|_{g(t)}$ does not pass through zero. 

\begin{remark}
The conclusion of Corollary \ref{CorAncient} when $c=\tfrac14|\Wd|^2_{\sst0}$ is consistent with the results of Theorem \ref{JHermFlow}, which was proved under weaker assumptions. 
\end{remark}

\smallskip

When $c \geq \frac12 |\Wd|_{\sst0}^2$, from corollaries \ref{CorEternal} and \ref{CorImmortal} we deduce the next result, which is consistent with the general result \cite[Cor.~2]{FPPZ}. 
\begin{corollary}
Let $\f(t)$ be the solution \eqref{solspec} to the source-free Type IIA flow starting at a special Type IIA geometry $(\omega,\fz)$ such that $c \geq \frac12 |\Wd|_{\sst0}^2$. 
Then, $\f(t)$ exists for all positive times and 
\[
\lim_{t\rightarrow+\infty} |N_{J_t}|_{g(t)}^2 = 0. 
\]
In particular, the symplectic form $\omega$ admits compatible almost complex structures with arbitrary small Nijenhuis tensor. 
\end{corollary}

From Proposition \ref{SpecialFR}, Proposition \ref{lemmapsimt} and Appendix \ref{appendix}, we see that  
Corollary \ref{CorEternal} holds for the examples of special SHF structures on the Lie algebras $A_{5,7}^{-1,-1,1}\oplus\R$, $A_{5,17}^{a,-a,1}\oplus\R$ and $\frg_{6,118}^{0,-1,-1}$,  
Corollary \ref{CorAncient} holds for the examples on $\fre(1,1)\oplus\fre(1,1)$ and $\frg_{6,54}^{0,-1}$, 
and Corollary \ref{CorImmortal} holds for the examples on $\frg_{5,1}\oplus\R$, $\frg_{6,N3}$ and $\frg_{6,38}^{0}$. 
As a consequence, all compact locally homogeneous spaces corresponding to these Lie algebras admit a solution to the Type IIA flow which is described by Theorem \ref{SolThm} and its corollaries.

\begin{remark}
The solution to the type IIA flow on the nilmanifold studied in  \cite[Sect.~9.3.2]{FPPZ} corresponds to the example of special SHF structure on the nilpotent Lie algebra $\frg_{5,1}\oplus\R$.  
In that case, $\Fz=2$ and $c =  |\Wd|_{\sst0}^2$, whence $a(t) = \Fz^3\, t = 8t$. 
\end{remark}

\bigskip\noindent
{\bf Acknowledgements.}
The author was supported by GNSAGA of INdAM, and by the project PRIN 2017  ``Real and Complex Manifolds: Topology, Geometry and Holomorphic Dynamics''.  
He would like to thank Anna Fino, Fabio Podest\`a, Luigi Vezzoni and Lucio Bedulli for useful conversations.

\appendix
\section{Examples of special SHF structures on Lie algebras}\label{appendix}
In this appendix, we describe the special SHF structures on unimodular Lie algebras considered in \cite{FiRa} and mentioned in Proposition \ref{SpecialFR}. 
For each Lie algebra $\frg$, we consider the structure equations given in Theorem \ref{solvableSHF} with respect to a basis $\mathcal{B}^* = (e^1,\ldots,e^6)$ of $\frg^*$, 
and we use this basis to write the forms $\omega$, $\psip$, $\Wd$ and the metric $g$. 
The expression of the almost complex structure $J$ can be deduced from the identity $\omega = g(J\cdot,\cdot)$, while the expression of $\psim$ can be obtained computing $\psim = *_g\psip$. 

According to \eqref{dWdexpr}, the Chevalley-Eilenberg differential of $\Wd = -*_g d *_g \psip$ is given by
\[
d\Wd = \frac{\left|\Wd\right|^2}{4}\,\psip +\gamma,
\]
where $\gamma\in \Lambda^3_{12}\frg^*$ is of the form $\gamma = S_*\psip$, for a certain $S\in\Sym^{\sst-}(\frg)$.  
We write both $\gamma$ and the matrix representation of $S$ with respect to the basis $\mathcal{B} = (e_1,\ldots,e_6)$.  
For the examples having $S\neq0$, we also specify a basis of $\frg^*$ that is adapted to the special SHF structure. \medskip

\begin{enumerate}[$\bullet$]

\item $\fre(1,1)\oplus\fre(1,1)	=	(0,-e^{13},-e^{12},0,-e^{46},-e^{45})$:  \vspace{0.1cm}
\begin{enumerate}[]
\item $\omega	= e^{14}+e^{23}+2\,e^{56}, \quad   \psip	= e^{125}-e^{126}-e^{135}-e^{136}+e^{245}+e^{246}+e^{345}-e^{346}, $ \vspace{0.1cm}
\item $g	= (e^1)^2+(e^2)^2+(e^3)^2+(e^4)^2+2\,(e^5)^2+2\,(e^6)^2,$ \vspace{0.1cm}
\item $\Wd= 2\left(e^{26}+e^{25}+e^{36}-e^{35}\right),$ \vspace{0.1cm}
\item $\gamma	= 0,\quad  S = 0.$ 
\end{enumerate}
\medskip

\item $\frg_{5,1}\oplus\R	=	(0,0,0,0,e^{12},e^{13})$: 
\vspace{0.1cm}
\begin{enumerate}[]
\item $\omega	= e^{14} +e^{26}+e^{35}, \quad \psip	= e^{123}+e^{156}+e^{245}-e^{346},$ \vspace{0.1cm}
\item $g	= (e^1)^2+(e^2)^2+(e^3)^2+(e^4)^2+(e^5)^2+(e^6)^2,$ \vspace{0.1cm}
\item $\Wd = e^{26}-e^{35},$ \vspace{0.1cm}
\item $\gamma	= \tfrac32\,e^{123} -\tfrac12\,e^{156}-\tfrac12\,e^{245} +\tfrac12\,e^{346}, \quad S= \diag\left(-\tfrac12,-\tfrac12,-\tfrac12,\tfrac12,\tfrac12,\tfrac12\right),$ \vspace{0.1cm}
\item adapted basis $\left(e^1,e^4,-e^3,-e^5,e^2,e^6\right)$.
\end{enumerate} 
\medskip

\item $A_{5,7}^{-1,-1,1}\oplus\R	= (e^{15},-e^{25},-e^{35},e^{45},0,0)$:
\vspace{0.1cm}
\begin{enumerate}[]
\item $\omega	=  -e^{13}+e^{24}+e^{56}, \quad \psip	= -e^{126}-e^{145}-e^{235}-e^{346},$ \vspace{0.1cm}
\item $g	= (e^1)^2+(e^2)^2+(e^3)^2+(e^4)^2+(e^5)^2+(e^6)^2,$ \vspace{0.1cm}
\item $\Wd = 2\left(e^{14}-e^{23}\right), $ \vspace{0.1cm}
\item $\gamma	= 2\left(e^{126}-e^{145}-e^{235}+e^{346}\right), \quad S	= \diag\left(0,0,0,0,-2,2\right),$\vspace{0.1cm}
\item adapted basis $\left(e^3,e^1,e^2,e^4,e^5,e^6 \right)$. 
\end{enumerate}
\medskip

\item $A_{5,17}^{a,-a,1}\oplus\R = (a e^{15}+e^{25},-e^{15}+a e^{25},-a e^{35}+e^{45},-e^{35}-a e^{45},0,0),~a>0$:
\vspace{0.1cm}
\begin{enumerate}[]
\item $\omega	= e^{13}+e^{24}+e^{56}, \quad \psip	=  e^{125}-e^{146}+e^{236}-e^{345},$ \vspace{0.1cm}
\item $g	= (e^1)^2+(e^2)^2+(e^3)^2+(e^4)^2+(e^5)^2+(e^6)^2,$ \vspace{0.1cm}
\item $\Wd = -2a\left(e^{12}+e^{34}\right),$ \vspace{0.1cm}
\item $\gamma	= 2a^2\left(e^{125} + e^{146} - e^{236} - e^{345} \right), \quad S =\diag\left(0,0,0,0,-2a^2,2a^2\right),$ \vspace{0.1cm}
\item adapted basis $\left(e^1,e^3,e^2,e^4,e^5,e^6 \right)$. 
\end{enumerate}
\medskip

\item $\frg_{6,N3}	= (0,0,0,e^{12},e^{13},e^{23})$: 
\vspace{0.1cm}
\begin{enumerate}[]
\item $\omega	= 2\, e^{16}+e^{25}-e^{34}, \quad \psip = -e^{123}+e^{145}-2e^{246}-2\,e^{356},$ \vspace{0.1cm}
\item $g	= (e^1)^2+(e^2)^2+(e^3)^2+(e^4)^2+(e^5)^2+4\,(e^6)^2,$ \vspace{0.1cm}
\item $\Wd = 4\,e^{16}-e^{25}+e^{34},$ \vspace{0.1cm}
\item $\gamma	= -\tfrac92\,e^{123} - \tfrac32\, e^{145} - \tfrac32\,e^{246}-\tfrac32\,e^{356}, \quad S	= \diag\left(-\tfrac32,-\tfrac32,-\tfrac32,\tfrac32,\tfrac32,\tfrac32 \right),$ \vspace{0.1cm}
\item adapted basis $\left(e^1,2\,e^6,e^3,-e^4,e^2,e^5 \right)$.
\end{enumerate}
\medskip

\item $\frg_{6,38}^{0}		= (e^{23},-e^{36},e^{26},e^{26}-e^{56},e^{36}+e^{46},0)$:
\vspace{0.1cm}
\begin{enumerate}[]
\item $\omega	=  -2\,e^{16}+e^{34}-e^{25}, \quad \psip =  -2\,e^{135}-2\,e^{124}+e^{236}-e^{456}, $ \vspace{0.1cm}
\item $g	= 4\,(e^1)^2+(e^2)^2+(e^3)^2+(e^4)^2+(e^5)^2+(e^6)^2, $ \vspace{0.1cm}
\item $\Wd = 4\,e^{16}-e^{25}+e^{34},$ \vspace{0.1cm}
\item $\gamma	= 3\,e^{124}+3\,e^{135} +\tfrac92\,e^{236}+\tfrac32\,e^{456}, \quad S = \diag\left(\tfrac32,-\tfrac32,-\tfrac32,\tfrac32,\tfrac32,-\tfrac32 \right),$ \vspace{0.1cm}
\item adapted basis $\left(e^6,2\,e^1,e^3,e^4,-e^2,e^5 \right)$.
\end{enumerate}
\medskip

\item $\frg_{6,54}^{0,-1}	= (e^{16} + e^{35}, -e^{26} + e^{45}, e^{36}, -e^{46}, 0, 0)$: 
\vspace{0.1cm}
\begin{enumerate}[]
\item $\omega	= e^{14}+e^{23}+\sqrt{2}\,e^{56}, \quad \psip = e^{125}-\sqrt{2}\,e^{136}+\sqrt{2}\,e^{246}+e^{345},$ \vspace{0.1cm}
\item $g	= (e^1)^2+(e^2)^2+(e^3)^2+(e^4)^2+(e^5)^2+2\,(e^6)^2,$ \vspace{0.1cm}
\item $\Wd =  \sqrt{2}\,e^{13}-e^{14}+e^{23}+\sqrt{2}\,e^{24},$ \vspace{0.1cm}
\item $\gamma	= -\tfrac32\,e^{125} - \tfrac{\sqrt{2}}{2}\, e^{136} +  \tfrac{\sqrt{2}}{2}\, e^{246} + \tfrac12\,e^{345}, \quad S = \diag\left(\tfrac12,\tfrac12,-\tfrac12,-\tfrac12,\tfrac12,-\tfrac12 \right),$ \vspace{0.1cm}
\item adapted basis $\left(e^1,e^4,e^2,e^3,e^5,\sqrt{2}\,e^6 \right)$.
\end{enumerate}
\medskip

\item $\frg_{6,118}^{0,-1,-1}	= (-e^{16} +e^{25},-e^{15} -e^{26}, e^{36} -e^{45}, e^{35} +e^{46}, 0, 0)$:
\vspace{0.1cm}
\begin{enumerate}[]
\item $\omega	= e^{14}+e^{23}-e^{56},  \quad \psip	= e^{126}-e^{135}+e^{245}+e^{346},$ \vspace{0.1cm}
\item $g = (e^1)^2+(e^2)^2+(e^3)^2+(e^4)^2+(e^5)^2+(e^6)^2,$ \vspace{0.1cm}
\item $\Wd = 2\left(e^{12}-e^{34}\right),$ \vspace{0.1cm}
\item $\gamma	= 2\left(e^{126} + e^{135} - e^{245} + e^{346} \right), \quad  S= \diag(0,0,0,0,2,-2),$ \vspace{0.1cm}
\item adapted basis $\left(e^1,e^4,e^3,-e^2,-e^5,e^6 \right)$.
\end{enumerate}

\end{enumerate}


\end{document}